\documentclass[pdflatex, sn-mathphys-ay]{sn-jnl}
\usepackage{mathrsfs}
\usepackage{graphicx, subcaption}%
\usepackage[compatibility=false]{caption}%
\usepackage{multirow}%
\usepackage{amsmath,amssymb,amsfonts}%
\usepackage{amsthm}%
\usepackage{mathrsfs}%
\usepackage[title]{appendix}%
\usepackage{xcolor}%
\usepackage{textcomp}%
\usepackage{manyfoot}%
\usepackage{booktabs}%
\usepackage{algorithm}%
\usepackage{algorithmicx}%
\usepackage{algpseudocode}%
\usepackage{listings}%
\usepackage{natbib}%
\usepackage[normalem]{ulem}

\newcommand{\reals}{\mathbb{R}}
\newcommand{\intg}{\mathbb{Z}}
\newcommand{\ind}{\mathbb{I}}

\newcommand{\mc}{\mathcal}

\newcommand{\mscr}{\mathscr}
\renewcommand{\v}[1]{\boldsymbol{#1}}
\newcommand{\pp}{\mathbb{P}}
\newcommand{\dd}{\mathrm{d}}
\newcommand{\ee}{\mathbb{E}}

\newcommand{\cB}{B}

\newcommand{\poi}{\textsf{Poi}}

\newcommand{\mx}{\v{x}}
\newcommand{\mX}{\v{X}}

\newcommand{\scG}{\mathscr{G}}
\newcommand{\lt}{\left}
\newcommand{\rt}{\right}
\newcommand{\wt}{\widetilde}
\newcommand{\wh}{\widehat}

\newcommand{\var}{\mathbb{V}\mathrm{ar}}
\newcommand{\vol}{\mathsf{Vol}}
\newcommand{\maxd}{\mathsf{MD}}
\newcommand{\mcc}{\mathsf{MCC}}
\newcommand{\ntg}{\mathsf{NTG}}
\newcommand{\ec}{\mathsf{EC}}
\newcommand{\mcs}{\mathsf{MCS}}

\newtheorem{observation}{Observation}
\newtheorem{corollary}{Corollary}
\newtheorem{lemma}{Lemma}
\newtheorem{experiment}{Numerical Study}

\theoremstyle{thmstyleone}%
\newtheorem{theorem}{Theorem}
\newtheorem{proposition}[theorem]{Proposition}%

\theoremstyle{thmstyletwo}%
\newtheorem{example}{Example}%
\newtheorem{remark}{Remark}%

\theoremstyle{thmstylethree}%

\begin{document}

\title[Rare-Event Simulation for Random Geometric Graphs]{Efficient Rare-Event Simulation for Random Geometric Graphs via Importance Sampling}

\author*[1]{\fnm{Sarat} \sur{Moka}}\email{s.moka@unsw.edu.au}

\author[2]{\fnm{Christian} \sur{Hirsch}}\email{hirsch@math.au.dk}

\author[3]{\fnm{Volker} \sur{Schmidt}}\email{volker.schmidt@uni-ulm.de}

\author[4]{\fnm{Dirk} \sur{Kroese}}\email{kroese@maths.uq.edu.au}

\affil*[1]{\orgdiv{School of Mathematics and Statistics}, \orgname{University of New South Wales}, \orgaddress{ \city{Sydney}, \postcode{2052}, \country{Australia}}}

\affil[2]{\orgdiv{Data Science and Statistics}, \orgname{Aarhus University}, \orgaddress{\city{Aarhus}, \postcode{8000}, \country{Denmark}}}

\affil[3]{\orgdiv{Institute of Stochastics}, \orgname{Ulm University}, \orgaddress{ \city{Ulm}, \postcode{D-89069}, \country{Germany}}}

\affil[4]{\orgdiv{School of Mathematics and Physics}, \orgname{The University of Queensland}, \orgaddress{\city{Brisbane}, \postcode{4072}, \country{Australia}}}

\abstract{Random geometric graphs defined on Euclidean subspaces, also called Gilbert graphs, are widely used to model spatially embedded networks across various domains. In such graphs, nodes are located at random in Euclidean space, and any two nodes are connected by an edge if  they lie within a certain distance threshold. Accurately estimating rare-event probabilities related to key properties of these graphs, such as the number of edges and the size of the largest connected component, is important in the assessment of risk associated with catastrophic incidents, for example. However, this task is computationally challenging, especially for large networks. Importance sampling offers a viable solution by concentrating computational efforts on significant regions of the graph. This paper explores the application of an importance sampling method to estimate rare-event probabilities, highlighting its advantages in reducing variance and enhancing accuracy. Through asymptotic analysis and numerical studies, we demonstrate the effectiveness of our methodology, contributing to improved analysis of Gilbert graphs and showcasing the broader applicability of importance sampling in complex network analysis.}


\maketitle

\section{Introduction}\label{sec:intro}

Random geometric graphs have emerged as powerful mathematical models for representing spatially embedded networks in various fields such as wireless communication, sensor networks, and materials science, see e.g.  \cite{cov, baccelli2009stochastic1, baccelli2009stochastic2, bm, netCom,kenniche2010random, thm}. 
These graphs are defined by distributing nodes at random in a metric space, connecting pairs of nodes within a certain distance threshold, and forming a network that captures spatial relationships. In particular, in this paper, we consider random geometric graphs on a $d$-dimensional subset of the Euclidean space $\reals^d$ for any fixed integer $d\ge 1$, where a graph is created by a collection of (random) points with an edge between any two points that are within the (Euclidean) distance  of one length unit from each other. Such a graph is also known as  Gilbert graph \citep{gilbert}. For a wide and deep mathematical treatment of Gilbert graphs, we refer the reader to \cite{penrose}.

Accurate estimation of key characteristics in Gilbert graphs, such as the mean values of  the size of the largest component or the maximum degree, is crucial for understanding the behavior and performance of systems modeled by these graphs. However, except for a few isolated special cases, such fundamental characteristics cannot be computed in closed form. On the other hand, Monte Carlo simulation has become a central tool in the study of such networks; see e.g.  \cite{bb2,bb1}, where the simplicity of the Gilbert graph model makes it possible to estimate the typical behavior of large random networks by standard Monte Carlo simulation to a very high precision. However, in many applications, it is not enough to know the average case. We need to understand how the system behaves with respect to rare events, where a rare event is an event that occurs infrequently but can have a significant impact. For instance, in the context of telecommunication networks, it is not enough that the network provides good service on average; rather, users expect it to work well with very high probability. These challenges have been the motivation for questions of rare-event probabilities in spatial random networks, which is now a vibrant research field. {While the network models encountered in practice are more complex than the Gilbert graph, we believe that our results represent a first step toward rare-event analysis of more realistic spatial random networks, with significant potential in combining the importance-sampling schemes developed in this paper with large-deviation principles.}

%
%
To understand the scope and complexity of such research problems, we first note that investigating rare events is a challenging topic even for basic random graph models that do not involve any form of geometric information, such as the Erd\H{o}s-R\'enyi graph. For example, there is a series of papers that investigate probabilities of  large deviations for the number of triangles in this type of graphs, see \cite{tr1,tr3,tr2}. Also, the largest connected component was recently investigated, see \cite{comp1,comp2}. While \cite{comp1}  {  deals with only} the Erd\H{o}s-R\'enyi graph, the more recent paper by \cite{comp2} considers more general (kernel-based)  random  graphs.

%
%
However, for Gilbert graphs,  rare events are  harder to analyze. For instance, while it is easy to analyze the probability of having atypically many edges in the Erd\H{o}s-R\'enyi graph, this problem is difficult in the case of Gilbert graphs  \citep{harel}. Loosely speaking, the large deviations are governed by a condensation effect, i.e., the most likely reason for observing too many edges is to have a larger number of nodes 
in a small area that gives rise to a clique. In contrast, the behavior in the regime of lower large deviations is completely different. There, the most likely reason to observe too few edges comes from consistent changes throughout the sampling window \citep{lowTails}.  Only very recently it became possible to understand the rare-event behavior of the largest connected component in spatial models of complex networks \citep{comp3}.

In the light of these challenges, the overall aim of our paper is to show the effectiveness of the powerful technique of importance sampling for the purpose of rare-event analysis in the context of the Gilbert graph. More precisely, we pursue the following goals:
 (i) We propose conditional Monte Carlo and importance sampling estimators for a variety of rare events in the Gilbert graph such as the question whether edge count, maximum degree or clique count are below a fixed threshold. (ii) While these estimators are of a general abstract form, we present a specific grid-based scheme, which we show is easy to implement.
(iii)  Our paper is the first one which gives an estimator with bounded relative error in the context of the Gilbert graph. More precisely, in Theorem~\ref{thm:asym_Yhat} we show that  under mild conditions,  for a fixed sampling window both our proposed conditional Monte Carlo and the importance-sampling estimator have bounded relative error. (iv) Finally, we illustrate this  in a scaling regime of a growing window, where we can prove that the importance sampling estimator exhibits a bounded relative error, whereas the conditional Monte Carlo does not. This illustrates that the more complicated importance-sampling scheme holds the promise of more substantial reduction in variance, where the theoretical results are also supported by an extensive simulation study.

To summarize this, we can state that
importance sampling offers a promising approach to address the computational challenges associated with estimating properties of Gilbert graphs. By assigning appropriate weights to samples, importance sampling focuses computational effort on regions of the graph that contribute significantly to the desired property, thus improving the efficiency of estimation.
The present paper explores the application of importance sampling techniques for estimating key graph properties in Gilbert graphs, where we explore the theoretical foundations of importance sampling, emphasizing its benefits in reducing variance and improving the accuracy of estimators. Additionally, we discuss the intricacies of adapting importance sampling to the specific characteristics of these graphs, considering factors such as spatial distribution, distance metrics, and connectivity constraints.
Through analysis and simulations, we demonstrate the effectiveness of importance sampling in providing more efficient and accurate estimates of critical graph properties. We specifically identify two key regimes to illustrate asymptotic efficiency of the proposed importance sampling estimator. Our findings not only contribute to the methodological toolbox for analyzing random geometric graphs but also shed light on the broader applicability of importance sampling in the realm of complex network analysis.

The subsequent sections of this paper are organized as follows. In Section~\ref{sec:notation}, we introduce some notation that is useful throughout the paper. The problem setup of rare-event simulation for Gilbert graphs along with important examples is presented in Section~\ref{sec:prelim}. In Section~\ref{sec:est_meth}, we summarize two existing methods for rare-event simulation, namely na\"{i}ve and conditional Monte Carlo methods, and then introduce the general framework of our importance sampling approach. The focus of Section~\ref{sec:is} is put on implementation of the importance sampling method using blocking regions on the sampling window. In Section~\ref{sec:efficiency}, we compare the variances of all the methods and study {  the asymptotic efficiency} of the proposed method over two important regimes. Simulation results are presented in Section~\ref{sec:sim}, whereas Section~\ref{sec:conclusion} concludes.

\section{Notation and Efficiency Notions}
\label{sec:notation}
Throughout the paper, the underlying probability space is denoted by $(\Omega, \mscr{F}, \pp)$. 
The sets of real numbers and integers are denoted by $\reals$ and $\intg$, respectively, while the sets of non-negative real numbers and non-negative integers are correspondingly denoted by $\reals_+$ and $\intg_+$. 
For any probability measure $\mu$ and random element $\mX$, we write $\mX \sim \mu$ to denote that $\mX$ is distributed according to $\mu$. The distribution of a Poisson random variable with rate parameter $\beta>0$ is denoted by $\textsf{Poi}(\beta)$.
For a real-valued random variable $X$, its expectation and variance are denoted by $\ee[X]$ and $\var(X)$, respectively. 
When necessary, to emphasize the dependency on a measure $\mu$, we use the notation $\pp_\mu$, $\ee_\mu$ and $\var_\mu$ to make it clear that  probability, expectation and variance, respectively, are taken under the  measure $\mu$. For any fixed integer $d\ge 1$ and any Borel set $S \subseteq \reals^d$, we denote its volume by $\vol(S)$ which is, formally speaking, the $d$-dimensional Lebesgue measure of the set $S$.
In particular, we denote the volume of the $d$-dimensional Euclidean sphere of unit radius by~$\upsilon_d$, i.e.,
\begin{align}
\upsilon_d =  \frac{\pi^{d/2}}{\Gamma(d/2 + 1)},
\label{eqn:unit-sphere-volume}
\end{align} 
where {  $\Gamma: (0, \infty) \to \reals_+$} is the gamma function. 

For the asymptotic analysis considered in Section~\ref{sec:asymptotic}, we use two standard notions of efficiency. Suppose that $\lt\{Y_t : t > 0\rt\}$ is a family of (real-valued) estimators parameterized by $t > 0$ such that $\lim_{t \to \infty} \ee[Y_t] = 0$.  We say that the family has an asymptotic  bounded relative error as $t \rightarrow \infty$ if 
\begin{align}
	\limsup_{t \rightarrow \infty} \frac{\var\lt(Y_t\rt)}{\ee[Y_t]^2} < \infty. \label{eqn:bdd-rel-err}
\end{align}
A slightly weaker notion is  logarithmic efficiency, which holds if 
\begin{align}
	\liminf_{t \rightarrow \infty} \frac{|\log \var\lt(Y_t\rt)|}{2 |\log \ee[Y_t] |} \geq 1,\quad
\text{or, equivalently,} \quad
    \limsup_{t \rightarrow \infty} \frac{\var\lt(Y_t\rt)}{\ee[Y_t]^{(2-\varepsilon)}} = 0, \label{eqn:log-eff}
\end{align}
for each $\varepsilon > 0$. Here, ``weaker" means that logarithmic efficiency implies an asymptotic bounded relative error. Since $\lim_{t \to \infty} \ee[Y_t] = 0$, the variance $\var\lt(Y_t\rt)$ in Eqs.~\eqref{eqn:bdd-rel-err} 
and \eqref{eqn:log-eff}
can be replaced by the second moment $\ee\lt[Y_t^2\rt]$. For more details on these notions of efficiency, we refer to \cite{AG07} and \cite{RK2017}.

\section{Rare Events in Gilbert Graphs}
\label{sec:prelim}
In this section, we introduce the notions of   Gilbert graphs and of rare events associated with this type of graphs. We further provide some  examples of rare events. 

Consider the $d$-dimensional sampling window ${W = [0, \lambda]^d}$,  for some $\lambda>0$ and some integer  $d\ge 1$.  Then, for each ${n \in \intg_+}$, let $\scG_n$ be the family of all finite subsets of size $n$ on $W$, i.e.,
\[
\scG_n = \Big\{\mx = \{x_1, x_2, \dots, x_n\} : x_i \in W\,\, \text{for all}\,\, i=1,\ldots, n \Big\},
\]
where $n = 0$ corresponds to the empty set. Put ${\scG = \bigcup_{n \in \intg_+} \scG_n}$ and 
notice that the elements of $\scG$ are so-called  simple point patterns, i.e., they do not have multiple points.

A  point process is a random element ${\mX : (\Omega, \mc F) \to (\scG, \mc G)}$, where $\mc G$ denotes the Borel $\sigma$-algebra on $\scG$. 
By $\rho:\mc G\to[0,1]$ we denote the probability measure on $\scG$ under which, for each $n \in \intg_+$, 
the restriction $\mX_n = \{X_1, \dots, X_n\} \in \scG_n$ of $\mX$  to $\scG_n$  is a point process which consists of $n$ independent and uniformly distributed points in the window $W$.  
Furthermore, a point process $\mX = \{X_1, \dots, X_N\}\in \scG$ is called a $\kappa$- homogeneous Poisson point process  on $W$ 
with intensity ${\kappa > 0}$ if for the  (random) total number of points $N$ it holds that ${N \sim \mathsf{Poi}(\kappa \lambda^d)}$, and for any $n \in \intg_+$, conditioned on $N = n$, we have $\mX \sim \rho$.

From now onward, to simplify the notation, we put
\begin{align}
\beta = \kappa\, \vol(W) = \kappa \lambda^d,
\qquad\text{ and }\qquad
q_n = \exp(-\beta) \frac{\beta^n}{n!} \quad\mbox{for each \,$n \in \intg_+$}. 
\label{eqn:defn-beta}
\end{align}
That is, $\{q_n : n \in \intg_+\}$ is the probability mass function of the Poisson distribution $\poi(\beta)$. Furthermore, by $F_{\poi}:\intg_+\to [0,1]$ we denote the cumulative  distribution function of $\mathsf{Poi}(\beta)$, i.e.,
\begin{align}
F_{\poi} (n) = \sum_{i = 0}^n q_i  \quad\mbox{for each $n\in\intg_+$.}
\label{eqn:Fpoi}
\end{align}

For any $\mx \in \scG$, let $G(\mx)$ be the graph constructed by taking the points in $\mx$ as nodes and connecting every two distinct points ${x, x' \in \mx}$ by an edge if and only if $\|x - x' \| \leq 1$, where  $\| \cdot \|$ denotes the Euclidean norm in $\reals^d$. 
A random graph $G(\mX)$ is called a  Gilbert graph if the set of nodes $\mX = \{X_1, \dots, X_N\}$ constitutes a $\kappa$-homogeneous Poisson point process in $W$ for some $\kappa>0$. Then, for any  $B\in\mc G$,   it holds that
\begin{align}
\label{eqn:pm}
	\pp\lt(\mX \in B\rt) = \sum_{n \in \intg_+} q_n\, \pp_{\rho}(\mX_n \in B),
\end{align}
with $\mX_n=\{X_1,\ldots,X_n\} \sim \rho$ for each $n \in \intg_+$.
 Two realizations of Gilbert graphs on a bounded subset of the Euclidean plane $\reals^2$ are shown in Figure~\ref{fig:GG-example}.

  \begin{figure}[H]
    \centering
    \begin{subfigure}{0.48\linewidth}
        \centering
        \includegraphics[width=\linewidth]{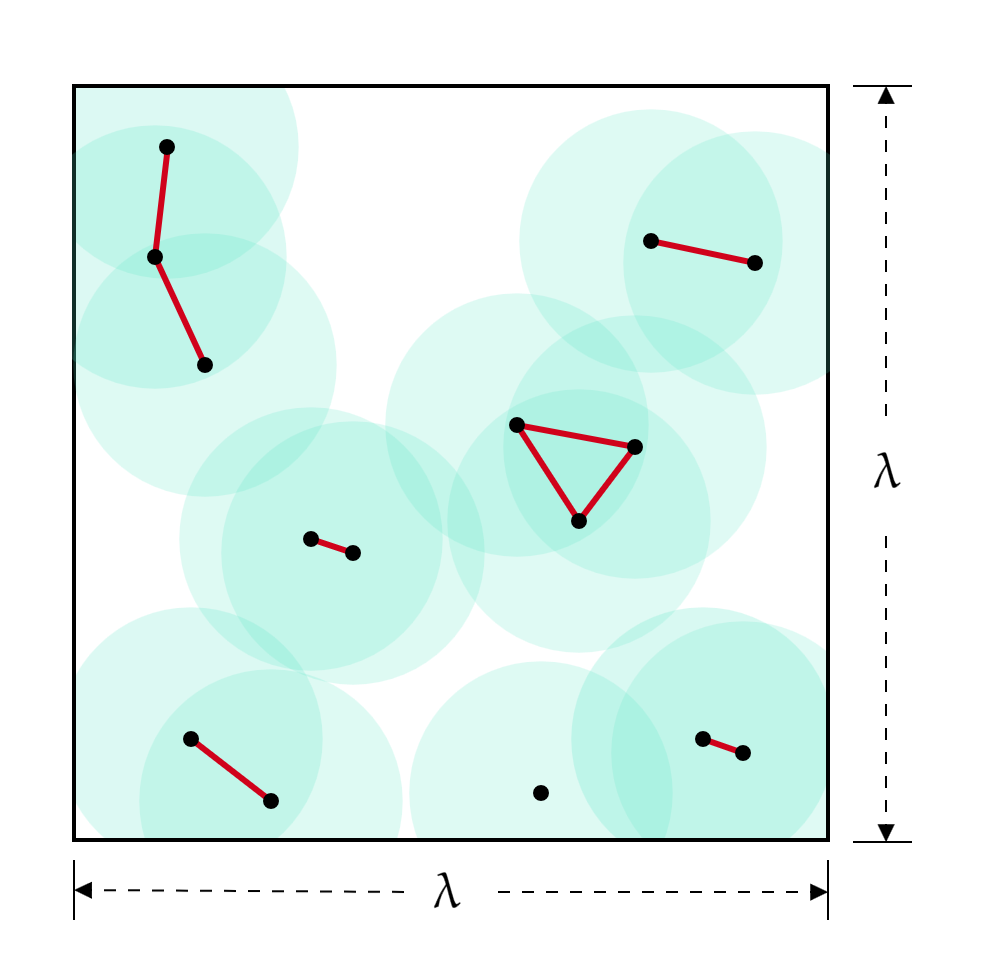}
        \caption{}
    \end{subfigure}
    \hfill
    \begin{subfigure}{0.48\linewidth}
        \centering
        \includegraphics[width=0.965\linewidth]{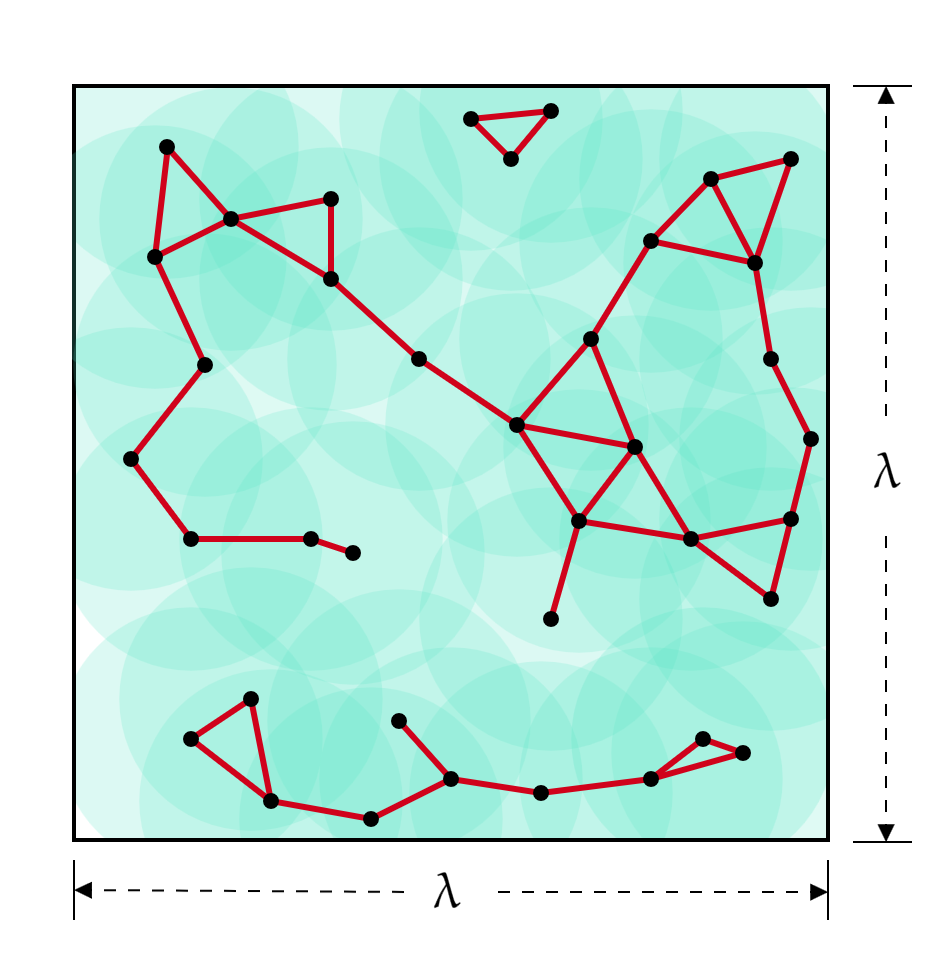}
        \caption{}
    \end{subfigure}
    \caption{Example realizations of Gilbert graphs on a 2-dimensional window $W = [0, \lambda]^2$, where black points represent the nodes, red lines represent the edges, and each circle centered at a node has a unit radius. Small intensity $\kappa$ typically leads to few nodes and few edges as in (a) while large $\kappa$ typically leads to a bigger graph with more edges as in (b).}
    \label{fig:GG-example}
\end{figure}

Consider a $\kappa$-homogeneous Poisson point process $\mX$ in $W$   and let $A \in  \mc G$ be a non-empty set of realizations of $\mX$ such that  $\pp(\mX \in A)$ is close to zero, i.e., the occurrence of $A$ is very unlikely. Then,  $\{\mX \in A\}$ is called a rare event. However, note that this definition contains a certain degree of vagueness from a mathematical point of view.
Furthermore, assume that $A\in \mc G$ satisfies the  hereditary property, i.e., 
for any $\mx, \mx' \in \scG$ such that $\mx \subseteq \mx'$, $\mx' \in A$ implies that $\mx \in A$.
A consequence of the hereditary property is that for any sequence of points $x_1, x_2, \ldots \in W$, it holds that 
\begin{align}
\label{eqn:hereditary}
	\ind(\mx_n \in A) \geq \ind(\mx_{n+1} \in A) \qquad\text{for each $n\ge 1$,}
\end{align}
where $\mx_n = \{x_1, \dots, x_n\}$ and $\ind(B):\Omega\to\{0,1\}$ denotes the indicator of the
 event $B\in\mscr{F}$, i.e., $\ind(B)(\omega)=1$ if $\omega\in B$, and $\ind(B)(\omega)=0$ otherwise.

When $A$ is taken to be the set of all the configurations of the Gilbert graph with no edges, the corresponding rare event probability $\pp(\mX \in A)$ appears as the grand partition function of the popular hard-spheres model in grand canonical form. This model has many applications in various disciplines, including  physics, chemistry, and material science; see e.g.  
\cite{krauth2006statistical, MJM17}. In particular the probability density of the hard-spheres model is given by $
f(\mx) = \ind(\mx \in A)/\pp(\mX \in A)$,  $\mx \in \scG$, and efficient estimation of $\pp(\mX \in A)$ is crucial for understanding key properties of the model \citep{doge2004grand}.

We will now give five more examples of sets $A\in\mc G$ that satisfy the hereditary property, and describe situations where the probability  $\pp(\mX \in A)$ is close to zero. 
Note that, if we take the threshold parameter $\ell = 0$ in these examples, then  the  probability~ $\pp(\mX \in A)$ is the same for  the first four examples, 
being equal to the grand partition function of the hard-spheres model.

\begin{example} [\bf Edge Count]
\label{expl:ec}
\normalfont
For any $\mx \in \scG$,  the number of edges in $G(\mx)$  will be denoted by $\ec(\mx)$. Furthermore, for a given threshold $\ell \in \intg_+$, let 
$A = \{\mx \in \scG : \ec(\mx) \leq \ell \}$ be the event of interest. 
Then, the  value of ~$ \pp\lt(\mX \in A\rt)$, i.e., the probability that the number of edges in the Gilbert graph $G(\mX)$ is at most $\ell$, can be very small for values of $\kappa$ and $\ell$ such that $\ell$ is much smaller 
than the expected number of edges ${\ee\lt[ \ec(\mX)\rt]}$.
\end{example}

\begin{example} [\bf Maximum Degree]
\label{expl:deg} 
\normalfont

We say that two nodes of a graph are adjacent if there is an edge between them.
For any $\mx \in \scG$, the  degree of a node $x \in \mx$ of $G(\mx)$, denoted by $\mathsf{Deg}(x)$, is the number of nodes  $x' \in \mx$  adjacent to $x$, i.e., such that $0<\|x-x'\|\le 1$. 
The maximum degree of the graph $G(\mx)$ is given by
$
\maxd\lt( \mx\rt) = \max\{\mathsf{Deg}(x) : x \in \mx\}.
$
Consider the event
$A = \{ \mx \in \scG : \maxd\lt( \mx\rt) \leq \ell\}$ that the maximum degree is less than or equal to $\ell$, for some  $\ell\in\intg_+$.  
Then, for values of $\kappa$ and $\ell$ such that $\ell$ is much smaller than the expected maximum degree $\ee[\maxd \lt( \mX\rt)]$, the   probability ~$\pp\lt(\mX \in A\rt)$ can be very small.
\end{example}

\begin{example}[\bf Maximum Connected Component]
\normalfont
\label{expl:mcc}
For any $\mx \in \scG$, two nodes $x, x' \in \mx$ of $G(\mx)$ with $x\not= x'$ are said to be  connected if there is a sequence of nodes $x_1, x_2, \dots, x_n \in \mx$ for some integer $n>1$ such that 
$x_1 = x$, $x_n = x'$, and for all $k = 2,\dots, n$, $x_k$ and $x_{k -1}$ are adjacent to each other. 
A subset of nodes $\mx' \subseteq \mx$ of $G(\mx)$ is called a  connected component if all the nodes in $\mx'$ are 
connected with each other and none of the nodes in $\mx'$ is connected with any node in $\mx\setminus \mx'$. 
Let $\mcc(\mx)$ be 
the size of the largest connected component in $G(\mx)$, where the size of a connected component is the number of nodes belonging to this connected component. 
Consider the event 
$A = \{ \mx \in \scG : \mcc\lt( \mx\rt) \leq \ell +1\}$ 
for some $\ell  \in \intg_+$ such that $\ell$ {  is much smaller} than $\ee[\mcc(\mX)]$. 
Then, also in this case,  the  probability ~$\pp\lt(\mX \in A\rt)$ 
that the size of each connected component is at most~$\ell + 1$, can be very small.  
\end{example}

\begin{example}[\bf Maximum Clique Size]
\normalfont
\label{expl:mcs}
A  clique of a graph is a subgraph that is complete, which means that any two distinct vertices in the subgraph 
are adjacent. Denote the 
 maximum clique size by $\mcs(\mx)$ and    consider the event 
$A = \{ \mx \in \scG : \mcs( \mx) \leq \ell + 1\}$,
for some $\ell  \in \intg_+$ which is much smaller than $\ee[\mcs( \mX)]$. 
The probability~$ \pp\lt(\mX \in A\rt)$ 
 that all  cliques are at most of  size~$\ell+1$, is then  close to zero. 
\end{example}

\begin{example}[\bf Number of Triangles]
\normalfont
\label{expl:ntg}
Another important quantity is the number of triangle subgraphs in an undirected graph, where 
a triangle is a clique with $3$ vertices and $3$ edges. 
Let $\ntg(\mx)$ denote the number of triangles in $G(\mx)$, and consider the event
$A = \{ \mx \in \scG : \ntg\lt( \mx\rt) \leq \ell \}$
for some  $\ell  \in \intg_+$. 
Then, the  probability~$ \pp\lt(\mX \in A\rt)$  that {  the number of triangles} in $G(\mX)$ is not larger than $\ell$, can be very small if  $\ell$ is 
much smaller than
$\ee[\ntg(\mX)]$. Note that  
instead of triangles, we can  consider cliques of any fixed size $k\ge 3$ and the event   that the number of cliques of size $k$ is at most $\ell$. 
\end{example}

{ 
\begin{example}[\bf Planarity]
\normalfont
\label{expl:planar}
A graph is called \emph{planar} if it can be embedded in the plane without edge crossings.
Consider the event
$A = \{ \mx \in \scG : G(\mx) \text{ is planar} \}$.
Since any subgraph of a planar graph is planar, the event $A$ satisfies the hereditary property; refer to \cite{diestel2024} for more details. By Kuratowski's theorem~\citep{kuratowski1930probleme}, a graph is non-planar if and only if it contains a subdivision of $K_5$ (the complete graph on five vertices) or $K_{3,3}$ (the complete bipartite graph with two sets of three mutually connected vertices) as a subgraph.
For sufficiently large intensity $\kappa$, the graph $G(\mX)$ almost surely contains a $K_5$ or $K_{3,3}$ minor and the probability $\pp(\mX \in A)$ is close to zero.
\end{example}
}

\section{Estimation Methods}
\label{sec:est_meth}

In this section, we first review two existing methods for estimating  probabilities  of the form given in Eq.~\eqref{eqn:pm}, namely  na\"{i}ve Monte Carlo and  conditional Monte Carlo. 
Then, we present a general framework of the proposed  importance sampling method, which consistently exhibits a variance less than or equal to that of the two existing methods. The following simple observation helps in understanding why the proposed method is efficient. 
\begin{observation}
\label{obs:fund}
For any event $B\in\mscr{F}$, let $\xi_1,\xi_2:\Omega\to[0,1]$ be two random variables such that ${\xi_1 = \ind(B)}$ and $\ee[\xi_2] = \ee[\xi_1] = \pp(B)$. Then, $\var(\xi_2) \leq \var(\xi_1),$
which easily follows from the fact that $\xi_2^2 \leq \xi_2$ almost surely and $\ee[\xi_2] = \ee[\xi_1]=\ee[\xi_1^2]$. 
This essentially suggests that for estimating the  probability $\pp(B)$, instead of using samples of $\xi_1$, it can be more efficient to use samples of $\xi_2$ whenever possible.
\end{observation}

For any fixed  $B\in\mscr{F}$,
a simple and basic method for estimating the probability
$ \pp(B)$ 
 is  na\"{i}ve Monte Carlo.
To see this, let $\mX$ be a $\kappa$-homogeneous Poisson point process in the window $W$ and consider ${B=\{\mX \in A\}}$ for some 
 $A \in  \mc G$ and the random variable
\begin{align}
 Y = \ind(\mX \in A).
 \label{eqn:naive-Y}
\end{align}
Now, for any integer $m\ge 1$, let $Y_1, \dots, Y_m$ be a sequence of independent and identically distributed (iid) copies of $Y$. Then, the sample mean 
\begin{align}\label{nai.est.def}
Z_{\tt NMC} = \frac{1}{m}(Y_1 + \cdots + Y_m)
\end{align}
is an unbiased estimator of  $\pp(\mX \in A)$; i.e., $ \ee_\rho[Z_{\tt NMC}] = \ee_\rho[Y] = \pp(\mX \in A)$. Note that it is easy to generate  a sample of $Y$. Namely, one only needs to generate a sample
$\mx \in \scG$ 
of the $\kappa$-homogeneous Poisson point process $\mX$ in the window $W$, 
and take $Y = 1$ if $\mX \in A$, otherwise, take $Y = 0$.

On the other hand, for the edge count problem (see Example~\ref{expl:ec}), a  
conditional Monte Carlo estimator was recently proposed by \cite{HMTK20}. We now state this method more generally in our set-up. Suppose that $X_1, X_2, \ldots:\Omega\to W$ is a sequence of independent  random points that are uniformly distributed in $W$.  Let $\mX_n = \{X_1, \ldots, X_n\}$ for each $n  \in \intg_+$, and 
$
	M = \max \{n  \in \intg_+: \ind(\mX_n \in A) = 1\}
$ for some 
 $A \in  \mc G$ which satisfies the hereditary property.
Then, from Eq.~\eqref{eqn:pm} we get that
\begin{align}
\pp(\mX \in A) = \ee_\rho\lt[ \sum_{n \in \intg_+} q_n\, \ind(\mX_n \in A) \rt]= \ee\lt[ \sum_{n = 0}^M q_n \rt]
      = \ee\lt[ F_{\poi}(M) \rt],
	\label{eqn:cmc_exp} 
\end{align}
where $F_{\poi}$
is the cumulative distribution function of $\poi(\beta)$
given in Eq.~\eqref{eqn:Fpoi}.
Now, consider the random variable $\wh Y:\Omega\to[0,1]$  which is defined as
\begin{align}
\label{eqn:Y_hat}
\wh Y = \sum_{n \in \intg_+} q_n\, \ind(\mX_n \in A) =F_{\poi}(M),
\end{align} 
and,  for any integer $m\ge 1$, let $\wh Y_1, \dots, \wh Y_m$ be a sequence of iid copies of $\wh Y$. 
Then, from Eq.~\eqref{eqn:cmc_exp} we have
\begin{align}
\label{cmc.est.def}
Z_{\tt CMC} = \frac{1}{m}(\wh Y_1 + \cdots + \wh Y_m)
\end{align}
is also an unbiased estimator of  $\pp(\mX \in A)$.
Furthermore, by  Observation~\ref{obs:fund}, it is evident that the variance of the conditional Monte Carlo estimator $Z_{\tt CMC}$ does not exceed that 
of the na\"{i}ve Monte Carlo estimator $Z_{\tt NMC}$ given in Eq.~\eqref{nai.est.def}, because for the random variables $Y$ and $\wh Y$ introduced in Eqs.~\eqref{eqn:naive-Y}
and \eqref{eqn:Y_hat}, respectively, it holds that
$Y \in \{0,1 \}$ and $\widehat Y \in [0,1]$, while $\ee[Y] = \ee[\widehat Y]$; see also Proposition~\ref{prop:var_relation} later in Section~\ref{sec:efficiency}. 

However, note that the estimator $\widehat Y$ given in Eq.~\eqref{eqn:Y_hat}
is still a weighted sum of Bernoulli random variables. Using importance sampling, 
we now construct an estimator for  $\pp(\mX \in A)$ with a possibly further reduced variance,  where each Bernoulli random variable 
$\ind(\mX_n \in A)$ in  Eq.~\eqref{eqn:Y_hat} will be replaced by a non-binary random variable, which takes values in the interval $[0, 1]$ and has  the same expectation. 

Recall that in both methods considered above, i.e., for getting the   na\"{i}ve and conditional Monte Carlo estimators for  $\pp(\mX \in A)$, the random points  $X_1, X_2, \ldots:\Omega\to W$ were independent and uniformly distributed in the cubic sampling window $W = [0, \lambda]^d$. Instead of this sampling scheme, i.e., instead of generating each of these points independently of the other points, 
we now use a different procedure, where each point generation can depend on the locations of already generated points. 
Formally, we are no longer considering
the  probability measure $\rho:\mc G\to[0,1]$ 
introduced in Section~\ref{sec:prelim},
under which 
the restriction $\mX_n = \{X_1, \dots, X_n\} \in \scG_n$ of the $\kappa$-homogeneous Poisson point process $\mX$  to $\scG_n$  is
a point process that consists of $n$ independent and uniformly distributed points 
for each $n \in \intg_+$. 
But, instead of $\rho$, we consider a new probability measure $\mu:\mc G\to[0,1]$  {  such that $\rho$ is absolutely continuous with respect to $\mu$ on  $A \in  \mc G$}, i.e., 
\begin{align}
\rho(A')=\int_{A'}
 L(\mx) \mu(\dd\mx),
	\label{eqn:likelihood} \qquad\mbox{for each $A'\in  {\mc G }\cap A $.}
\end{align}
{  Since $\widetilde{Y}$
  is an unbiased estimator for $\mathbb{P}(\boldsymbol{X} \in A)$ under $\mu$,
  from Observation~1, requiring $L \leq 1$ ensures $\widetilde{Y} \in [0,1]$, which guarantees
  $\mathrm{Var}(\widetilde{Y}) \leq \mathrm{Var}(Y)$.} We then get 
\begin{align}
	\pp_\rho(\mX_n \in A) = \ee_\mu[\ind(\mX_n \in A) L(\mX_n)], \quad n\ge 1. \label{eqn:change_of_measure}
\end{align}
This equality can be used in order to construct a third unbiased estimator for $\pp(\mX \in A)$, besides the estimators $Z_{\tt NMC}$ and $ Z_{\tt CMC}$ discussed above. For this, let
\begin{align}
\label{eqn:Y-tilde} 
\wt Y &= \sum_{n \in \intg_+} q_n\, \ind(\mX_n \in A)L(\mX_n),
\end{align} 
where it follows from Eqs.~\eqref{eqn:pm} and~\eqref{eqn:change_of_measure} that 
\(\pp(\mX \in A)= \ee_\mu[\wt Y].\)
Furthermore, for any integer $m\ge 1$, let $\wt Y_1, \dots, \wt Y_m$ be a sequence of independent and identically distributed copies of $\wt Y$. 
Then,  
\begin{align}\label{def.wt.zet}
Z_{\tt IS} = \frac{1}{m}(\wt Y_1 + \cdots + \wt Y_m)
\end{align}
is an unbiased estimator for the probability  $\pp(\mX \in A)$.

Proposition~\ref{prop:var_relation} in Section~\ref{sec:efficiency} establishes the relationship between the variances of the three estimators $Z_{\tt NMC}$, $Z_{\tt CMC}$ and $Z_{\tt IS}$, showing that the variance of $ Z_{\tt IS}$ does not exceed the variance of~$Z_{\tt CMC}$, which in turn does not exceed the variance of~$Z_{\tt NMC}$, where it is our goal  to select $\mu$ so that the event $\{\mX \in A\}$ is not rare under $\mu$, and hence from now on we refer to $\mu$ as the importance sampling measure. 

\section{Importance Sampling Using Blocking Regions}
\label{sec:is}
In Section~\ref{sec:est_meth}, we presented a general idea of importance sampling to estimate the probability~$\pp(\mX \in A)$, where $A \in  \mc G$ is some rare-event of interest. 
We now present an example of an importance sampling measure $\mu$, where the choice of  $\mu$ is inspired by the perfect sampling method for hard-spheres models proposed by~\cite{MJM17}. 

The key idea of this importance sampling method is to generate points sequentially so that each point falls 
outside a certain blocking region created by the existing points. 
Specifically, for any existing configuration ${\mx_n = \{x_1, \dots, x_n\}}$, 
we say that a region ${\cB(\mx_n) \subseteq W}$ is blocked by $\mx_n$ if any  new point $x$ 
selected in $\cB(\mx_n)$ is guaranteed to satisfy ${\mx_n \cup \{x\} \notin A }$.
That is, selecting the next point over the blocked region results in a configuration outside~$A$. 
Note that $B(\mx)$, possibly the whole window $W$, is defined for any configuration $\mx \in \scG$, not just for configurations in $A$.

Since random points $X_1,X_2,\ldots$ generated under $\rho$  are independent and uniformly distributed in the cubic sampling window $W=[0,\lambda]^d$, we have that  
$\rho(\dd \mx) = (1/\lambda^{nd}) \dd\mx,$  for any $\mx \in \scG_n$ and
for all integers $n \ge 1$, where $\lambda^d$ is the volume of $W$. 
Suppose now that for every configuration ${\mx \in \scG}$, a blocking region $\cB(\mx)$ can be identified easily and its volume can be computed exactly. Then, under $\mu$,  random points $X_1, X_2, \dots $ are sequentially 
generated such that $X_{n}$ is uniformly generated on the non-blocking region ${W\setminus B(\mX_{n-1})}$, 
where ${\mX_{n-1} = \{X_1, \dots, X_{n-1}\}}$ is the set of already generated points. 
We stop the procedure when either ${\mX_{n} \notin A}$ or ${B(\mX_n) = W}$. 
Then, with $\mx_0$ denoting an empty set of points, the Radon–Nikodym derivative $L:A\to\mathbb{R}^+$ introduced in Eq.~\eqref{eqn:likelihood} is given by
\begin{align}
\label{eqn:radon-nikodym-der}
L(\mx_n) = \prod_{i=0}^{n-1}\Big(1  - \frac{\vol(\cB(\mx_{i}))}{\vol(W)}\Big),\qquad \mbox{for any $n\ge 1$ and $\mx_n \in A$,}
\end{align}
where $\cB(\mx_{0}) = \varnothing$; i.e., the blocking region is empty when there are no points. 
Note that the term $\left(1 - \vol(\cB(\mx_{i}))/\vol(W)\right)$  in Eq.~\eqref{eqn:Y-tilde}  is the ratio of the uniform density 
$1/\vol(W) = 1/\lambda^d$ on the whole window~$W$ and the uniform density $1/ (\lambda^d - \vol(\cB(\mx_{i})))$ over 
the non-blocked region~${W\setminus \cB(\mx_{i})}$ for the $(i+1)$th point.  
Given that $\vol(\cB(\mx)) \leq \lambda^d$ for any $\mx \in \scG$, we ensure that $L(\mx) \le 1$  for any $\mx \in A$, as desired.
With this choice of the importance sampling measure $\mu$, the estimator $Z_{\tt IS}$ for $\pp(\mX \in A)$, introduced in Eq.~\eqref{def.wt.zet}, is 
determined by $\wt Y$ given in Eq.~\eqref{eqn:Y-tilde} with $L:A\to[0,1]$ taken to be as in Eq.~\eqref{eqn:radon-nikodym-der}; see also
Algorithm~\ref{alg:is} below.
\begin{algorithm}
  \caption{Importance Sampling Method}
  \label{alg:is}
  \begin{algorithmic}[1]
    \State{$\cB \leftarrow \varnothing$}
    \State{$n \leftarrow 0$ and $L_0 \leftarrow 1$} 
	\Repeat
	\State{$n \leftarrow n+1$}
    	 \State{$L_n \leftarrow L_{n-1} \times\lt(1 - \frac{\vol(\cB)}{\vol(W)}\rt)$} 
    	 \State{Generate the next point $X_n$ uniformly on $W\setminus \cB$}    
	     \State{$\mX_n \leftarrow  \lt\{X_1, \dots, X_n\rt\}$}
	     \State{Update the blocking region $\cB$ for the next point}
    \Until{$\cB = W$ or $\mX_n \notin A$}
    \State $\wt Y \leftarrow \sum_{i = 0}^{n} q_i\, L_i$\\  
\Return{$\wt Y$}
  \end{algorithmic}
\end{algorithm}

Ideally, in each iteration of the algorithm, we would like to identify the blocking region with the maximum possible volume; 
see Figure~\ref{subfig:block_a}. {  In that case, we can remove $\mX_n \notin A$ from the stopping condition as it never happens when $\cB \neq W$}. Unfortunately, it can be computationally challenging to identify such a maximal blocking region
and to compute its volume exactly so that a uniformly distributed point can be generated over the region outside the blocking region. However,  since any subset of a maximal blocking region is also a valid blocking region, we can identify such sub-blocking regions in a computationally easy way. 
For this, we use a {\em grid} on $W$. 
In particular, we partition the window $W$ into a cubic grid of size~$K^d$. Each cell in the grid is uniquely indexed by a vector of dimensions $d$ $(k_1, \dots, k_d) \in \{0, 1, \dots, K-1\}^d$ so that $\times_{j = 1}^d\lt[k_j{\lambda/K}, \, (k_j+1){\lambda}/{K}\rt]$ is the cell with index $(k_1, \dots, k_d)$.

To implement Algorithm~\ref{alg:is}, the first point $X_1$ is generated uniformly on the window $W$. At the $n$-th iteration, suppose that $\mX_{n-1}$ denotes the set of points generated in the previous iterations. To generate the $n$-th point, we identify a set of cells that are completely covered by the maximum blocking region and take the blocking region $\cB(\mX_{n-1})$ as the union of these cells. Then, we generate the next point $X_n$ uniformly over the non-blocking cells. 
See Figure~\ref{fig:block} for an illustration of this grid-based approach to find blocking regions. 
\begin{figure}[ht]
    \centering
    \begin{subfigure}{0.49\linewidth}
        \centering
        \includegraphics[scale=0.375]{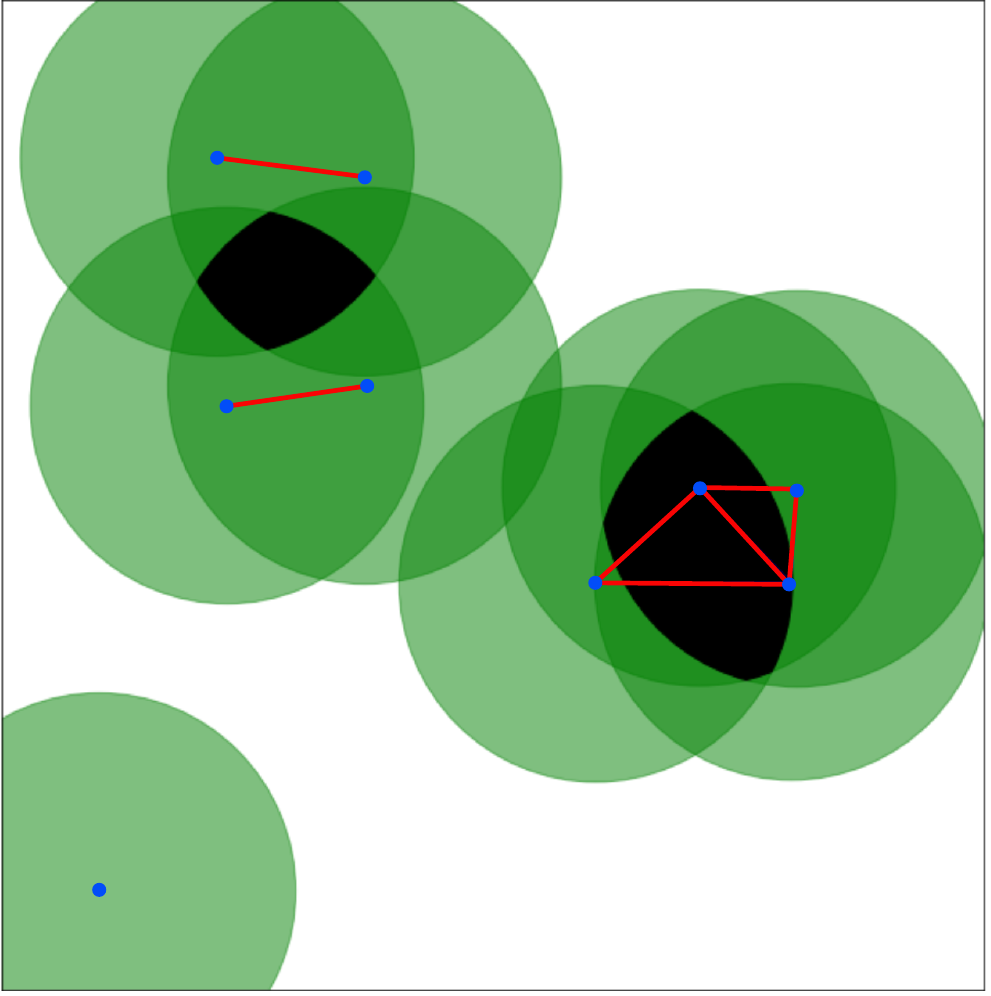}
        \caption{\label{subfig:block_a}}
    \end{subfigure}%
    ~ 
    \begin{subfigure}{0.49\linewidth}
        \centering
        \includegraphics[scale=0.38]{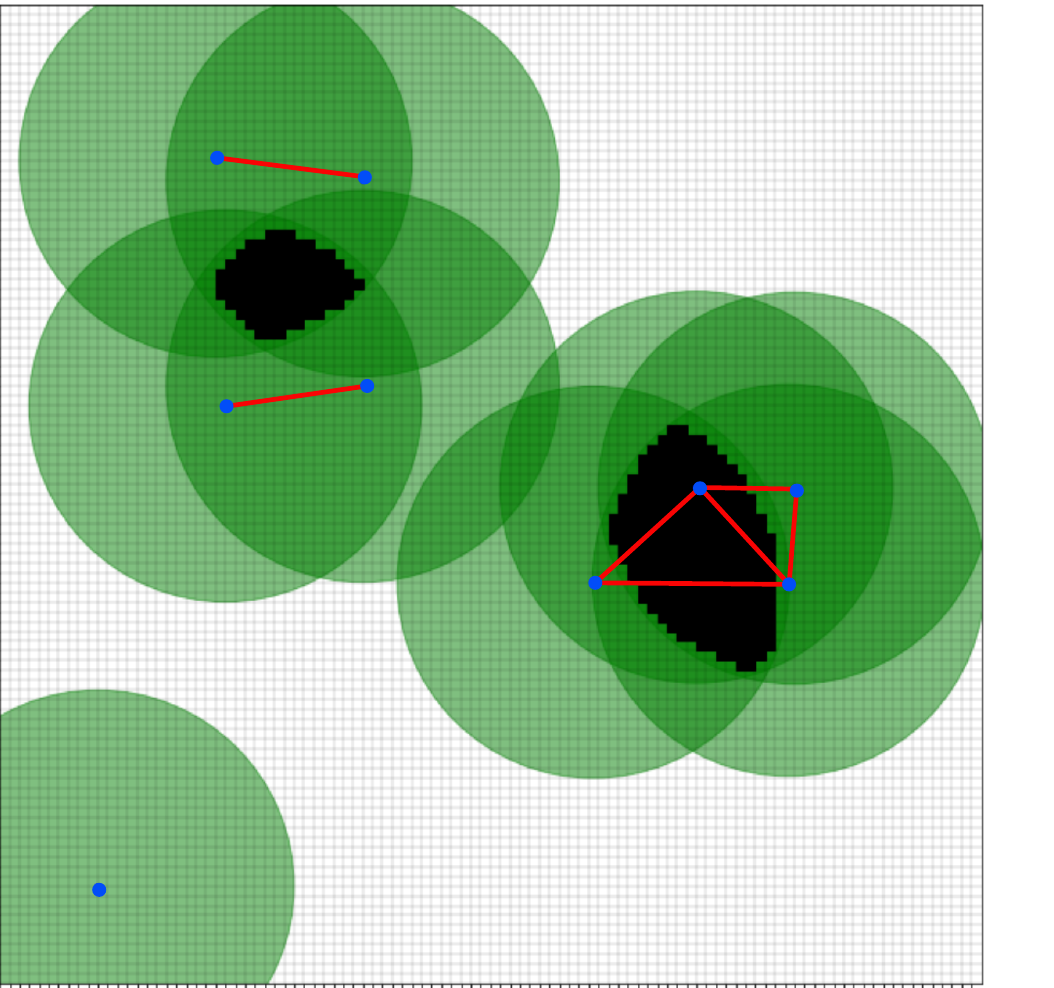}
        \caption{\label{subfig:block_b}}
    \end{subfigure}
\caption{Illustration of generating points under the proposed importance sampling method for the edge count problem with $\ell = 10$ on a two-dimensional window. There are $9$ existing points creating $7$ edges. The black region in (a) is the maximum possible blocking region and selecting the next point over the black region in  (a) will result in the number edges being more than $10$. Ideally, we would like to generate the next point outside this maximal blocking region (as in  (a)). However, identifying that region is difficult. The grid-based importance sampling method easily approximates this region from inside as shown in (b).}
    \label{fig:block}
\end{figure}

Below we describe procedures for constructing the blocking regions $\cB(\mX_1)$, $\cB(\mX_2), \dots$ for Examples~\ref{expl:ec}-\ref{expl:ntg} and, in red, for Example~\ref{expl:planar}.  
For this, we define the distance $\mathsf{Dist}(C, C')$ between any two distinct cells $C$ and $C'$ as
\begin{align}
\mathsf{Dist}(C, C') = \sup\big\{\|x - x'\| : x \in C, x'\in C'\big\},
\label{eqn:distance}
\end{align}
which is the  greatest possible distance between a point in $C$ and a point in $C'$.  
If ${\mathsf{Dist}(C, C') \leq 1}$, then the cells $C$ and $C'$ are called neighbors, denoted by $C \sim C'$.
For any point $x \in W$, $C(x)$ denotes the cell for which $x \in C(x)$. Observe that since we are generating each point $x$ uniformly on a cell, with probability $1$, $x$ must be an interior point of the cell, and thus $x$ belongs to only one cell.\\

\noindent
{\bf Edge Count.} Consider the rare-event probability defined in Example~\ref{expl:ec}: the probability of the event
$A = \{\mx \in \scG : \ec(\mx) \leq \ell \}$. That is, we want to estimate the probability that the number of edges $\ec(\mx)$ in the graph $G(\mx)$ is less than $\ell$. After generating $n$ points, suppose $\mx_n = \{x_1, \dots, x_n\} \in A$ is the current configuration.  A cell $C$ of the grid is said to be of {\em order $k$} if 
$\big| \{ i = 1, \dots, n : C \sim C(x_i) \}\big| = k$.
That is, exactly $k$ of $\{C(x_1), \dots, C(x_n)\}$ are within unit distance from $C$ under the distance definition Eq.~\eqref{eqn:distance}. 
Denote the union of all the order $k$ cells by $D_k(\mx_n)$. Let $e_n = \ec(G(\mx_n))$. Then, the region
\begin{align*}
\cB(\mx_n) = \bigcup_{k >  \ell - e_n} D_{k}(\mx_{n})
\end{align*}
is blocked because selecting the next point over $\cB(\mx_n)$ will prevent the event $A$ from occurring. 
In each iteration, after generating the next point $x$, we update the order of all the cells that are neighbors to $C(x)$.\\

\noindent
{\bf Maximum Degree.} For Example~\ref{expl:deg}, where $A = \{ \mx \in \scG : \maxd\lt( \mx\rt) \leq \ell\}$ and  $\maxd\lt( \mx\rt)$ denotes the maximum degree of the graph $G(\mx)$, the grid-based importance sampling procedure is similar to the procedure stated above, except a cell $C$ is blocked either if $C$ is of order greater than $\ell$ or if there is an existing point $x_i$ with degree at least $\ell$ and $C \sim C(x_i)$. This is because, in both the cases, a new point on $C$ will lead to {  the maximum degree of the graph being greater than $\ell$}. Thus, 
 for this example, the blocked region $\cB(\mx_n)$ is given by
\begin{align*}
\cB(\mx_n) = \lt(\bigcup_{k > \ell} D_{k}(\mx_n)\rt) \cup \lt( \bigcup_{x \in \mx_n} \{C : C \sim C(x)\,\, \text{and}\,\, \mathsf{deg}(x) \geq \ell  \}\rt),
\end{align*}
where $\mx_n$ is the existing configuration and $\mathsf{deg}(x)$ is the degree of the node at $x$.\\

\noindent
{\bf Maximum Connected Component.}  To construct the blocking region for the rare event $A = \{ \mx \in \scG : \mcc\lt( \mx\rt) \leq \ell +1\}$ 
in Example~\ref{expl:mcc}, for any configuration $\mx_n \in A$ after generating $n$ points, decompose the set of points into connected components. Consider all the connected components of size $\ell + 1$. We call a cell $C$ blocked if there exists $x \in \mx_n$ such that $x$ is part of a connected component of size $\ell +1$ and $C \sim C(x)$. Then, the overall blocked region $\cB(\mx_n)$ is the union of all the blocked cells.\\

\noindent
{\bf Maximum Clique Size.}  For Example~\ref{expl:mcs}, where $A = \{ \mx \in \scG : \mcs( \mx) \leq \ell + 1\}$, suppose that  $\mx_n \in A$ is the configuration after generating $n$ points. Similar to the above procedure for the maximum connected component, identify all the cliques in $\mx_n$ of size $\ell + 1$. We now call a cell $C$ blocked if there exists a clique $\mx' \in \mx_n$ of size $\ell+1$ such that $C \sim C(x)$ for all $x \in \mx'$.\\

\noindent
{\bf Number of Triangles.}
Finally, for Example~\ref{expl:ntg}, where $A = \{ \mx \in \scG : \ntg\lt( \mx\rt) \leq \ell \}$,  we generate points until there are exactly $\ell$ triangles.
After that, for each point generation, we identify the cells $C$ where a new point selection over $C$
results in a new triangle. The union of such cells is the blocking region for that point generation.\\

\noindent
{\bf Planarity.}                                                               
  For Example~\ref{expl:planar}, where $A = \{ \mx \in \scG : G(\mx) \text{ is
  planar}\}$,                                                                    
  consider any configuration $\mx_n \in A$ after generating $n$ points.
  A cell $C$ is blocked if a new point $x$ placed in $C$ would create a $K_5$ or
  $K_{3,3}$
  subgraph in the resulting graph.
  For $K_5$: $C$ is blocked if there exist four nodes
  $\{v_1, v_2, v_3, v_4\} \subseteq \mx_n$ forming a $K_4$ in $G(\mx_n)$,
  with $C \sim C(v_i)$ for $i = 1, 2, 3, 4$;
  any new point $x \in C$ is then connected to all four nodes,
  creating a $K_5$ on $\{x, v_1, v_2, v_3, v_4\}$.
  For $K_{3,3}$: $C$ is blocked if there exist nodes
  $\{u_1, u_2\} \subseteq \mx_n$ and $\{w_1, w_2, w_3\} \subseteq \mx_n$
  forming a $K_{2,3}$ in $G(\mx_n)$ (each $u_i$ connected to each $w_j$),
  with $C \sim C(w_j)$ for $j = 1, 2, 3$;
  any new point $x \in C$ is then connected to all of $w_1, w_2, w_3$,
  creating a $K_{3,3}$ with bipartition $\{x, u_1, u_2\}$ and $\{w_1, w_2,
  w_3\}$.

\begin{remark}[Graphs with a Fixed Number of Nodes]
\normalfont
Recall that the total number of nodes $N$ in the Gilbert graph is a Poisson random number. Now suppose that the graph is constructed with the number of nodes fixed, say $N = n$. Then, from Eq.~\eqref{eqn:Y_hat}, the conditional Monte Carlo estimator $\wh Y$ becomes identical to {  the na\"{i}ve Monte Carlo} estimator $Y$ given by Eq.~\eqref{eqn:naive-Y} (i.e., $F_\poi$ is replaced by the degenerative distribution with all the mass at $n$). That is 
\[
Y = \wh Y = \ind(\mX_n \in A),
\]
where for the configuration {  $\mX_n$ of $n$ points} it holds that $\mX_n \sim \rho$. Thus, the conditional Monte Carlo method brings no variance reduction in this scenario; i.e., $\var(\wh Y) = \var(Y)$.  On the other hand, the  proposed importance sampling can still reduce the variance, because
\[
\wt Y = \ind(\mX_n \in A)L(\mX_n),
\]
where $\mX_n \sim \mu$.
Since $L(\mX_n) \in [0, 1]$, we have $\var(\wt Y) < \var(Y)$, except for trivial cases with values of~$n$, where $L(\mX_n) = 1$ almost surely.
\end{remark}

\section{Efficiency Analysis}
\label{sec:efficiency}
In this section, we first demonstrate that our importance sampling estimator achieves the lowest variance among the three estimators presented in Section~\ref{sec:est_meth}. We then illustrate its asymptotic efficiency in comparison to the other methods through two interesting scenarios: one with a fixed sampling window and the other with a growing window. For this, to simplify the notation, let 
\begin{align}
    \label{eqn:defn_pn}
    p_n = \pp_\rho(\mX_n \in A)\quad \text{and}\quad  
    p^* = \pp(\mX \in A) = \sum_{n \in \intg_+} q_n p_n,
\end{align}
for a rare event $A$ of interest that satisfies the hereditary property, where we recall that $\mX$ denotes a $\kappa$-homogeneous Poisson point process on the sampling window $W = [0, \lambda]^d$. 
\subsection{Variance Comparison}
\label{sec:non-asymptotic}
Proposition~\ref{prop:var_relation} demonstrates that the variance of the importance sampling estimator is lower than that of the conditional Monte Carlo estimator, which, in turn, is lower than that of the na\"{i}ve Monte Carlo estimator. Here, the notation $\var_\mu(\wt Y)$ emphasizes that $\wt Y$ is constructed using points that are generated under the importance measure $\mu$, as in Eq. \eqref{eqn:Y-tilde},  while $\var_\rho(\wh Y)$ and $\var_\rho(Y)$ emphasize that both $\wh Y$  and $Y$ are constructed using points that are generated under $\rho$ as in Eq. \eqref{eqn:Y_hat} and Eq. \eqref{eqn:naive-Y}, respectively.

\begin{proposition}
\label{prop:var_relation}
For any intensity $\kappa$ and window size $\lambda$, we have
\[
\var_\mu(\wt Y) \leq \var_\rho(\wh Y) \leq \var_\rho(Y).
\]
\end{proposition}
\begin{remark}[Relationship with Optimal Importance Sampling]
    Using Theorem~1.2 in Chapter V of \citep{AG07}, we can conclude that the optimal (i.e., zero-variance) importance sampling measure $\mu^*$ for the rare-events considered in this paper has a Radon–Nikodym derivative  $L^*:A\to[0,1]$ given by
    \[
    L^*(\mx_n) = p_n \quad \text{for all}\, \, \mx_n \in A\,\, \text{and}\,\, n \in \intg_+,
    \]
    where $p_n$ is given in Eq.~\eqref{eqn:defn_pn}. 
    Unfortunately, sampling from such optimal measure is impractical as it involves the unknown probabilities  $p_n$. If we had access to all $p_n$, we could directly compute $p^* = \pp(\mX \in A)$ exactly using Eq.~\eqref{eqn:defn_pn}. 
    Our proposed importance sampling measure $\mu$ strikes a balance between practicality and variance reduction.  
    Specifically, the importance sampling estimator retains a positive, yet minimized, variance, as demonstrated in  Proposition~\ref{prop:var_relation}. 
    This is achieved by bringing the values of   $L(\mx_n)$ closer to  $L^*(\mx_n) = p_n$, as supported by Observation~\ref{obs:fund}. 
\end{remark}
\begin{proof}[Proof of Proposition~\ref{prop:var_relation}]
First observe that $
\ee_\rho[Y] = \ee_\rho[\wh Y] = \ee_\mu[\wt Y]$.
Therefore, it is sufficient to prove that 
\begin{align}
\label{eqn:2nd-mt-rel}
\ee_\mu[\wt Y^2] \leq \ee_\rho[\wh Y^2] \leq \ee_\rho[Y^2] = p^*,
\end{align}
where the equality $\ee_\rho[Y^2] = p^*$ holds, because $Y$ is a Bernoulli random variable.
The second inequality $\ee_\rho[\wh Y^2] \leq \ee_\rho[Y^2]$ follows from Observation~\ref{obs:fund}.
To prove the inequality $\ee_\mu[\wt Y^2] \leq \ee_\rho[\wh Y^2]$, note that 
\begin{align}
\ee_\rho[\wh Y^2] = \sum_{n \in \intg_+} \sum_{m \in \intg_+} q_n q_m\, \pp_\rho(\mX_n \in A\,\, \text{ and }\,\, \mX_m \in A)
             = \sum_{n \in \intg_+} q_n^2\, p_n  + 2 \sum_{ m < n } q_m q_n\, p_n,
             \label{eqn:expand_exp_Yhat}
\end{align}
where we used the fact that $\{\mX_n \in A\} \subseteq  \{\mX_m \in A\}$ for all $m < n$. Then, using Eq.~\eqref{eqn:change_of_measure}, we get that
\begin{align}
\ee_\mu[\wt Y^2] &= \ee_\mu\lt[ \lt(\sum_{n \in \intg_+} q_n \ind(\mX_n \in A) L(\mX_n) \rt)^2\rt]\nonumber\\
                          &= \sum_{n \in \intg_+} \sum_{m \in \intg_+}  q_n q_m \, \ee_\mu\lt[\ind(\mX_n \in A) \ind( \mX_m \in A) L(\mX_n) L(\mX_m)\rt]\nonumber\\                         
                          &= \sum_{n \in \intg_+} q_n^2\, \ee_\mu\lt[\ind(\mX_n \in A) L(\mX_n)^2 \rt] + 2 \sum_{ m < n } q_n q_m \, \ee_\mu\lt[\ind(\mX_n \in A) L(\mX_m) L(\mX_n)\rt].\label{eqn:tilde_Y_expand}
\end{align}
Since $0 \leq L(\mX_i) \leq 1$ for all $i$, this gives that 
\begin{align}                         
      \ee_\mu[\wt Y^2]    &\leq \sum_{n \in \intg_+} q_n^2\, \ee_\mu\lt[\ind(\mX_n \in A) L(\mX_n) \rt] + 2 \sum_{ m < n } q_nq_m \ee_\mu\lt[\ind(\mX_n \in A) L(\mX_n)\rt]\nonumber\\
                          &= \sum_{n \in \intg_+} q_n^2\,\pp_\rho(\mX_n \in A)  + 2 \sum_{ m < n } q_n q_m\,\pp_\rho(\mX_n \in A)\label{eqn:hat_Y_expand}
\end{align}
{  which is equal to $\ee_\rho[\wh Y^2 ]$ from \eqref{eqn:expand_exp_Yhat}, and hence completing the proof. }
\end{proof}

Proposition~\ref{prop:var_relation} established that in general our importance sampling method is 
superior to or, at the very least, as effective as the conditional Monte Carlo method in minimizing variance. 
Our next result, Proposition~\ref{prop:gamma_improve}, provides more insights on this aspect. 
To this end, for $m \leq n$, define
\begin{align}
\gamma_{m,n} = \ee_\rho\lt[L(\mX_{m})\,\big| \,\mX_n \in A\rt] = \frac{\ee_\rho[L(\mX_{m})\ind(\mX_n \in A)]}{p_n}. \label{eqn:gamma_mn}
\end{align}
That is, $\gamma_{m, n}$ is the conditional expectation of $L(\mX_m)$ given $\mX_n \in A$, where $\mX_n$ is a set of $n$~uniformly and independently 
distributed points on the observation window $W = [0, \lambda]^d$. 
\begin{proposition} 
\label{prop:gamma_improve}
It holds that
\begin{align}
\label{eqn:expand_hat_Y2}
\ee_\rho[\wh Y^2] = \sum_{n \in \intg_+} q_n^2 p_n + 2\sum_{m < n} q_n q_m p_n,
\end{align}
and 
\begin{align}
\label{eqn:expand_tilde_Y2}
\ee_\mu[\wt Y^2] = \sum_{n \in \intg_+} q_n^2 \gamma_{n,n}p_n + 2\sum_{m < n} q_n q_m \gamma_{m, n}p_n.
\end{align}
\end{proposition}

\begin{proof}
The expression in Eq.~\eqref{eqn:expand_hat_Y2} is established as Eq.~\eqref{eqn:hat_Y_expand} in the proof of Proposition~\ref{prop:var_relation}. To prove Eq.~\eqref{eqn:expand_tilde_Y2}, consider Eq.~\eqref{eqn:tilde_Y_expand}, and then using the definition of the Radon-Nikodym derivative $L$,  we can write for every $m \leq n$ that
\[
\ee_\mu\lt[\ind(\mX_n \in A) L(\mX_m) L(\mX_n)\rt] = \ee_\rho\lt[\ind(\mX_n \in A) L(\mX_m)\rt]. 
\]
Furthermore, using Eq.~\eqref{eqn:gamma_mn}, we obtain
\begin{align*}
\ee_\mu[\wt Y^2] &= \sum_{n \in \intg_+} q_n^2\, p_n\, \frac{\ee_\rho\lt[\ind(\mX_n \in A) L(\mX_{n}) \rt]}{p_n} + 2 \sum_{ m < n } q_n q_m \, p_n\, \frac{\ee_\rho\lt[\ind(\mX_n \in A) L(\mX_{m}) \rt]}{p_n},
\end{align*}
which is identical to Eq.~\eqref{eqn:expand_tilde_Y2} from the definition of $\gamma_{m, n}$ given by Eq.~\eqref{eqn:gamma_mn}.
\end{proof}

To see an implication of Proposition~\ref{prop:gamma_improve}, take $\ell = 0$ in any of the first 
four examples in Section~\ref{sec:prelim}. Then, $A$ is the set of all the hard-spheres configurations. If we further assume that the cell edge length, which is $\lambda/K$, in the grid is selected sufficiently smaller than $1$,
then as in \citep{MJM17}, for every  $\mx_n \in \scG_n \cap A$, we can show that
$\vol(\cB(\mx_n)) \geq n c \, \upsilon_d$,
for a positive constant $c$ with $\upsilon_d$ denoting the volume of a unit radius $d$-dimensional hyper-sphere given by Eq.~\eqref{eqn:unit-sphere-volume}. The value of $c$ can increase with the refinement of the grid used in importance sampling.
Therefore, using the definition of $L(\mX_m)$, for all $m, n$ with $m \leq n$, we have
\begin{align}
\label{eqn:gamma_upper}
\gamma_{m,n} \leq \prod_{i=0}^{m -1} \lt(1 - \frac{i c\, \upsilon_d}{\lambda^d} \rt)^+,
\end{align}
where $r^+ = \max(0, r)$ for any $r \in \reals$. 
Note from Bernoulli's inequality that $ 1 + rt \leq (1 + t)^r$ for all $r \geq 0$ and $t \geq -1$. Therefore, 
\begin{align*}
\gamma_{m,n} \leq \lt(1 - \frac{c \upsilon_d}{\lambda^d}\rt)^{\sum_{i = 1}^m (i-1)} = \lt(1 - \frac{c\upsilon_d}{\lambda^d}\rt)^{m(m-1)/2}.
\end{align*}
Thus, for any $n$, $\gamma_{m, n}$ decays with a rate faster than exponential in $m^2$ to reach zero for large~$m$. As a result, from Proposition~\ref{prop:gamma_improve}, $\ee_\mu[\wt Y^2]$ can be much smaller than $\ee_\rho[\wh Y^2]$, or equivalently,  
$\var_\mu(\wt Y)$ can be much smaller than $\var_\mu(\wh Y)$, as supported by the simulation results in Section~\ref{sec:sim}.

\subsection{Asymptotic Analysis}
\label{sec:asymptotic}
As mentioned above, in this analysis our goal is to theoretically illustrate the limiting performance of the proposed importance sampling method in two asymptotic regimes where $p^*$ approaches~$0$. In the first regime, with a fixed observation window, we show that both the conditional Monte Carlo estimator and the proposed importance sampling estimator are efficient, whereas in the second regime, with a growing window, we show that only the importance sampling estimator retains efficiency.

Towards this end, let $M$ be a integer-valued random variable with
\begin{align}
\pp(M = n) = \frac{q_n p_n}{\sum_{i \in \intg_+} q_i p_i}, \quad n \in \intg_+, \label{eqn:defn_M}
\end{align}
where $q_n$ and $p_n$ are defined in Eq.~\eqref{eqn:defn-beta} and Eq.~\eqref{eqn:defn_pn} respectively. 
In other words, $M$ is the number of points in a realization $\mX$ of the points of the Gilbert graph conditioned on $\mX \in A$.

Suppose an asymptotic regime is parameterized by a non-negative parameter $t \in \intg_+$ such that the rare-event probability $p^* = p^*(t)$ tends to $0$ as $t \to \infty$. 
As an example of such a regime, we can fix both the window size $\lambda$ and threshold $\ell$, and increase the intensity $\kappa$ to $\infty$ as $t \to \infty$. Another regime could be where the intensity $\kappa$ is constant and both $\ell$ and $\lambda$ go to $\infty$. In general, an asymptotic regime consists of changing combination of these parameters such that $p^*$ goes to zero asymptotically. 

It is easy to see that the na\"{i}ve Monte Carlo estimator $Y$, defined in Eq.~\eqref{eqn:naive-Y}, exhibits neither the bounded relative error nor the logarithmic efficiency over any regime where $\lim_{t \to \infty} p^* = 0$. This is true because $\var_\rho(Y) = p^*(1 - p^*)$ and 
\(
{\var_\rho(Y)}/{(p^*)^{(2- \epsilon)}} = {1}/{(p^*)^{(1 - \epsilon)}} - (p^*)^\epsilon 
\)
goes to $\infty$ as $t \to \infty$ for all $0 <\epsilon < 1$. Therefore, the na\"{i}ve Monte Carlo estimator is not efficient over any asymptotic regime where the rare-event probability $p^*$ goes to zero.
Thus, we focus only on the conditional Monte Carlo estimator $\wh Y$, defined in Eq.~\eqref{eqn:Y_hat}, and the importance sampling estimator $\wt Y$, defined in Eq.~\eqref{eqn:Y-tilde}. 

Recall that both $\wh Y$ and $\wt Y$ are unbiased estimators of $p^*$. As a consequence, from Proposition~\ref{prop:var_relation}, we can easily notice that in any regime, if $\wh Y$ has asymptotic logarithmic efficiency (respectively, bounded relative error), then $\wt Y$ also has asymptotic logarithmic efficiency (respectively, bounded relative error). 

\subsubsection{A Regime with a Fixed Sampling Window}
Theorem~\ref{thm:asym_Yhat} establishes the asymptotic efficiency of both $\wh Y$ and $\wt Y$ over an asymptotic regime.  Corollary~\ref{corr:case1} presents an interesting application of this theorem where the window $W$ and the threshold parameter $\ell$ are fixed and the intensity $\kappa$ increases unboundedly.
\begin{theorem}
\label{thm:asym_Yhat}
Consider an asymptotic regime parameterized by $t$, where $\lim_{t \to \infty} p^*  = 0$, and for  each~$t>0$, let $\mc{S}_t \subseteq \intg_+$ be the support of $M$ defined in Eq.~\eqref{eqn:defn_M} and denote its cardinality by $|\mc{S}_t|$.  Furthermore, suppose for all sufficiently large $t>0$ that $|\mc{S}_t| \leq n_0$ for some $n_0 < \infty$ and $\min_{n \in \mc{S}_t} p_n > 0$. Then, the conditional Monte Carlo estimator $\wh Y$ and the importance sampling estimators $\wt Y$ exhibit bounded relative error as $t \to \infty$.
\end{theorem}

Before proving Theorem~\ref{thm:asym_Yhat}, we first provide a specific corollary of the theorem.

\begin{corollary}
\label{corr:case1}
Consider the regime where $\lambda$ and $\ell$ are fixed. Then, for all examples stated  in Section~\ref{sec:prelim}, both $\wh Y$ and $\wt Y$ exhibit bounded relative error as the intensity $\kappa \to \infty$.
\end{corollary}

\begin{proof}
Note that $\kappa$ appears in the definition of Poisson distribution $\{q_n: n \geq 0\}$; see Eq.~\eqref{eqn:defn-beta}. Hence, $\kappa$ does not influence~$p_n$. 
From the hereditary property, $p_n$ is decreasing in $n$. Furthermore, since the window size $\lambda$ and threshold $\ell$ are fixed, for all  examples stated in Section~\ref{sec:prelim}, for sufficiently large values of $n$, every configuration $\mx \in \scG_n$ satisfies $\mx \notin A$. Thus, there exists $n_0$ such that $p_n = 0$ if and only if $n > n_0$. That is, the support of $M$ is $\{0, \dots, n_0\}$ for all $\kappa > 0$. Thus,  
\(
{  p_M \geq \min_{n = 0, \dots, n_0} p_n = p_{n_0} > 0},
\)
almost surely,
and thus we complete the proof of the corollary using Theorem~\ref{thm:asym_Yhat}.
\end{proof}

\begin{proof}[Proof of Theorem \ref{thm:asym_Yhat}]
We only need to show that $\wh Y$ has bounded relative error as it implies bounded relative error for $\wt Y$.  From Eq.~\eqref{eqn:bdd-rel-err}, we need to show that
$\limsup_{t \rightarrow \infty} \ee_\rho[\wh Y^2]/(p^*)^2 < \infty$. 
Now let $N$ and $N'$ be independent and identically Poisson distributed random variables, with distribution given by Eq.~\eqref{eqn:Fpoi}. Then, with the notation $a \vee b = \max(a, b)$, we get that
\begin{align*}
\ee_\rho[\wh Y^2] &= \sum_{n, m \in \intg_+} q_n q_m p_{n \vee m}\\
                &= \ee\lt[p_{N\vee N'}\rt]\\
                &= \ee\lt[p_{N} \ind(N' < N)\rt] + \ee\lt[p_{N'} \ind(N' \geq N)\rt]\\
                &= \ee\lt[p_{N} \ind(N' < N)\rt] + \ee\lt[p_{N} \ind(N' \leq N)\rt],
\end{align*} 
where in the second expectation of the last equality, we swapped $N$ and $N'$ as they are identically distributed. As a consequence, it follows that
\begin{align}
2\ee\lt[p_{N} \ind(N' < N)\rt] \leq \ee_\rho[\wh Y^2]\leq 2\ee\lt[p_{N} \ind(N' \leq N)\rt]. \label{eqn:ubd_lbd_yhat}
\end{align}
Let $n_t = \max \mc{S}_t$ be the maximum element of $\mc S_t$. 
From the hereditary property in Eq.~\eqref{eqn:hereditary}, for each $t$, $p_n \geq p_{n'}$ whenever $n < n'$. Therefore, $\mc{S}_t = \{0, 1, \dots, n_t\}$ and hence, $\pp(N \in \mc{S}_t) = \pp(N \leq n_t)$.
Note that for each $t$, we have
\begin{align}
\label{eqn:lbd_p*}
p^* = \ee[p_N] = \sum_{n \in \mc{S}_t} q_n p_n \geq p_{n_t} \, \pp(N \leq n_t),
\end{align}
and
\[
\ee\lt[p_{N} \ind(N \geq N')\rt] = \sum_{n \in \mc{S}_t} q_n p_n \pp(n \geq N') \leq \sum_{n \in \mc{S}_t} q_n \pp(n \geq N') \leq \pp(N \leq n_t)^2,
\]
where the first inequality holds because $p_n \leq 1$ and the second inequality is a simple consequence of the fact that $n_t$ is the maximum element of $\mc{S}_t$. Thus, 
\[
\frac{\ee_\rho[\wh Y^2]}{(p^*)^2} \leq \frac{1}{c^2} \frac{\pp(N \leq n_t)^2 }{\pp(N \leq n_t)^2} = \frac{1}{c^2} < \infty,
\]
where $c = \min_{n \in \mc{S}_t} p_n$ and the first inequality holds because $p_{n_t} \geq c$.   
\end{proof}

\subsubsection{A Regime with a Growing Sampling Window}
To demonstrate that over some regimes the importance sampling estimator $\widetilde{Y}$ can be efficient while the conditional Monte Carlo estimator $\widehat{Y}$ fails to be efficient, we now concentrate on the hard-sphere scenario, by setting $\ell = 0$ in Examples \ref{expl:ec} through \ref{expl:mcs} of Section~\ref{sec:prelim}. Note that for this hard-sphere scenario, $p^*$ is the probability that every pair of points in a $\kappa$-homogeneous Poisson point process on the window $W = [0, \lambda]^d$ is separated by at least a unit distance. Recall that $\beta$ is the expected number of nodes of the Gilbert graph, as defined in Eq.~\eqref{eqn:defn-beta}. For the asymptotic regime here, we assume that the volume of the window $W$ is $\beta^\delta$ (that is, $\lambda = \beta^{\delta/d}$) and the intensity $\kappa = \beta^{1 - \delta}$ for some $\delta > 1$, and study the asymptotic behaviors of $\wh Y$ and $\wt Y$ as $\beta \to \infty$. 
Note that, since $\delta > 1$, the volume of the window $W$ increases faster than the expected number of nodes of the graph. 
The limit point $\lim_{\beta \to \infty} p^*$ and the rate of convergence to the limit point vary depending on the value of $\delta$. In particular, Lemma~\ref{lem:lim-pstar} is a result from \cite{MJM17} and it implies that $p^*$ is a rare-event probability for $1 < \delta < 2$ as $\beta \to \infty$. For simplicity of analysis, we assume the ideal case where the importance sampling uses maximal-volume blocking regions.  
\begin{lemma}[Theorem~1 of \cite{MJM17}]
    \label{lem:lim-pstar}
    Suppose $\lambda = \beta^{\delta/d}$ with $\delta > 1$. Then, {  for the hard-sphere case (i.e., setting $\ell = 0$ in Examples \ref{expl:ec} through \ref{expl:mcs}),} 
    \[
    \lim_{\beta \to \infty}\frac{1}{\beta^{2 - \delta}}\log p^* = 
    - \frac{\upsilon_d}{2}, \quad \text{if} \,\, 1 < \delta < 2.
    \]
    and
    \[
    \lim_{\beta \to \infty} p^* = 
    \begin{cases}
        \exp(-\upsilon_d/2), &\quad \text{if} \,\, \delta = 2,\\
        1, &\quad \text{if} \,\, \delta > 2,
    \end{cases}
    \]
\end{lemma}
Theorem~\ref{thm:le_Ytilde} shows that the importance sampling estimator $\wt Y$ is efficient asymptotically as $\beta \to \infty$ while the conditional Monte Carlo estimate $\wh Y$ is inefficient.

\begin{theorem}
    \label{thm:le_Ytilde}
    Suppose $\lambda = \beta^{\delta/d}$ with $1 < \delta  < 2$. Then, {  for the hard-sphere case (i.e., setting $\ell = 0$ in Examples \ref{expl:ec} through \ref{expl:mcs}),}  as $\beta \to \infty$, the following is true:
    \begin{itemize}
        \item[(i)]  The importance sampling estimator $\wt Y$ exhibits logarithmic efficiency. 
        \item[(ii)] The conditional Monte Carlo estimator $\wh Y$ exhibits neither {  bounded relative error} nor logarithmic efficiency. 
    \end{itemize}
\end{theorem}
We prove statements (i) and (ii) of Theorem~\ref{thm:le_Ytilde} separately. The following Lemma~\ref{lem:upper-bound} is useful for establishing (i).

\begin{lemma}
\label{lem:upper-bound}
Suppose $\lambda = \beta^{\delta/d}$ with $1 < \delta  < 2$. Then, {  for the hard-sphere case (i.e., setting $\ell = 0$ in Examples \ref{expl:ec} through \ref{expl:mcs}),} 
\[
\sup_{n \leq 2\beta} \frac{\sqrt{\ee_\mu[L(\mX_n)^2]}}{p_n} = \exp\lt( O\lt(\beta^{2 - {\delta(d+1)/d}} + \beta^{3 - 2\delta} \rt)\rt).
\] 
\end{lemma}
In the lemma, $O(\cdot)$ denotes the standard big $O$ notation: A function $h(\beta)$ said to be $O(g(\beta))$ for a non-negative function $g(\beta)$ if ${\limsup_{\beta \to \infty} h(\beta)/g(\beta)}$ is finite. A proof of the lemma is  provided in Appendix~\ref{app:add-proofs}. 
 
\begin{proof}[Proof of Theorem~\ref{thm:le_Ytilde} (i)]
Because of  Eq.~\eqref{eqn:log-eff}, we need to show that 
\begin{align}
\lim_{\beta \to \infty} \frac{\ee_\mu[\wt Y^2]}{(p^*)^{2-\epsilon}} = 0, \label{eqn:log-eff-2}
\end{align}
for all $\epsilon > 0$. 
Under the importance sampling, we have $L(\mX_n) = 0$ if and only if $\mX_n \notin A$. Thus,
$\wt Y =  \sum_{n\in \intg_+} q_n L(\mX_n).$
Furthermore, using the Cauchy–Schwarz inequality, we get that
\begin{align*}
    \ee_\mu[\wt Y^2] 
    &= \sum_{n, m \in \intg_+} q_n q_m \ee_\mu[L(\mX_n) L(\mX_m)]\\
    &\leq \sum_{n, m \in \intg_+} q_n q_m \sqrt{\ee_\mu[L(\mX_n)^2]} \sqrt{\ee_\mu[L(\mX_m)^2]}
    = \lt(\sum_{n\in \intg_+} q_n \sqrt{\ee_\mu[L(\mX_n)^2]}\rt)^2.
\end{align*}
Thus, we establish Eq.~\eqref{eqn:log-eff-2} by showing that for any $\epsilon > 0$,
\begin{align}
    \lim_{\beta \to \infty} \lt[\frac{1}{(p^*)^{1 - \epsilon}}\sum_{n\in \intg_+} q_n \sqrt{\ee_\mu[L(\mX_n)^2]}\rt] = 0.
    \label{eqn:equiv-limit}
\end{align}
Now for each $\beta$, if we take $n_\beta = \max\{ n \in \intg_+: p_n \neq 0\}$, then $p_n \neq 0$ if and only if $n \leq n_\beta$. Since $p^* = \sum_{n\in \intg_+} q_n p_n$, with the notation $a \wedge b = \min(a, b)$, we have 
\begin{align}
\frac{1}{p^*} \sum_{n\in \intg_+} q_n \sqrt{\ee_\mu[L(\mX_n)^2]} &=
    \frac{1}{\sum_{n\in \intg_+} q_n p_n} \sum_{n\in \intg_+} q_n \sqrt{\ee_\mu[L(\mX_n)^2]}\nonumber\\ 
    &\leq \frac{1}{\sum_{n \leq  2\beta} q_n p_n} \sum_{n \leq  2\beta} q_n p_n \frac{\sqrt{\ee_\mu[L(\mX_n)^2]}}{p_n} \nonumber\\
    &\hspace{2cm} + \frac{1}{p^*}\sum_{n > 2\beta} q_n \sqrt{\ee_\mu[L(\mX_n)^2]}\nonumber\\
    &\leq \ee\lt[\frac{1}{p_M} \sqrt{\ee_\mu[L(\mX_M)^2]} \,\,\bigg|\,\, M \leq 2\beta\rt] 
    \nonumber\\ 
    &\hspace{2cm}
    + \frac{1}{p^*}\sum_{n > 2\beta} q_n \sqrt{\ee_\mu[L(\mX_n)^2]}, \label{eqn:split-expression}
\end{align}
where $M$ is distributed as Eq.~\eqref{eqn:defn_M}. From Lemma~\ref{lem:upper-bound}, we get that
\[
\ee\lt[\frac{1}{p_M}\sqrt{\ee_\mu[L(\mX_M)^2]} \, \bigg|\, M \leq  2\beta\rt] = \exp\lt( O\lt(\beta^{2 - {\delta(d+1)/d}} + \beta^{3 - 2\delta} \rt)\rt).
\]
From Lemma~\ref{lem:lim-pstar}, $
p^* = \exp\lt(-\frac{\upsilon_d \beta^{2 - \delta}}{2} (1 + o(1)) \rt),
$
and thus for any $\epsilon > 0$,
\begin{align}
    (p^*)^{\epsilon} \ee\lt[\frac{1}{p_M}\sqrt{\ee_\mu[L(\mX_M)^2]} \, \bigg|\, M \leq  2\beta\rt] &= \exp( -O(\beta^{2 - \delta})) \exp\lt( O\lt(\beta^{2 - {\delta(d+1)/d}} + \beta^{3 - 2\delta} \rt)\rt)\nonumber\\
    &\longrightarrow 0, \label{eqn:first-limit}
\end{align}
as $\beta \to \infty$, 
because $2 - \delta > 2- {\delta(d+1)/d}$ and $2 - \delta > 3 - 2\delta$. 
Furthermore, since $\pp(N > 2\beta) = \sum_{n > 2\beta} q_n$ is the tail probability of a Poisson random variable with mean $\beta$, from the Chernoff bound it follows that $\pp(N > 2\beta) \leq \exp(-c\beta)$ for some fixed constant $c > 0$; see, e.g., \cite{short2013improved}. Thus, for all $\epsilon>0$,
\begin{align}
\frac{\sum_{n >  2\beta} q_n \sqrt{\ee_\mu[L(\mX_n)^2]}}{(p^*)^{1- \epsilon}} &\leq \frac{\pp(N > 2\beta)}{(p^*)^{1- \epsilon}}\nonumber\\
&\leq  \exp(- c \beta) \exp\lt(\frac{\upsilon_d \beta^{2 - \delta}(1 - \epsilon)}{2} (1 + o(1)) \rt) 
\longrightarrow 0, 
\label{eqn:second-limit}
\end{align}
as $\beta \to \infty$, because $2 - \delta < 1$. From Eqs.~\eqref{eqn:first-limit}, \eqref{eqn:second-limit} and \eqref{eqn:split-expression}, we obtain Eq.~\eqref{eqn:equiv-limit} and thus Eq.~\eqref{eqn:log-eff-2}.
\end{proof}

\begin{proof}[Proof of Theorem~\ref{thm:le_Ytilde} (ii)]
Since the bounded relative error implies the logarithmic efficiency, it is enough to show that 
$\limsup_{\beta \to \infty} \lt(\ee_\mu[\wh Y^2]/(p^*)^{2-\epsilon}\rt) = \infty$ for some $\epsilon > 0$.
Let $N$ and $N'$ are iid Poisson random variables with distribution given by Eq.~\eqref{eqn:Fpoi}. Then,  
\begin{align*}
    \ee\lt[p_{N} \ind(N' < N)\rt] &\geq \ee\lt[p_{N} \ind(N' < N \leq \beta)\rt] 
    \geq p_{\lceil \beta \rceil}\pp(N' < N \leq \beta ),
\end{align*}
where $\lceil a \rceil$ denotes the smallest integer greater than $a$ and the last inequality follows from the fact that $p_n$ decreases with $n$. Therefore, from Eq.~\eqref{eqn:ubd_lbd_yhat}, for any $\epsilon > 0$, we get that
\begin{align}
\frac{\ee_\rho\lt[\wh Y^2\rt]}{(p^*)^{2 - \epsilon}} \geq 2 \frac{p_{\lceil \beta \rceil}}{(p^*)^{2-\epsilon}} \pp(N' < N \leq \beta ). \label{eqn:log-eff-yhat2}
\end{align}
We complete the proof by showing that $\pp(N' < N \leq \beta ) \to 1/8$ and $p_{\lceil \beta \rceil}/(p^*)^{2-\epsilon} \to \infty$ as $\beta \to \infty$.
Note that 
\[
\pp(\max(N', N) \leq \beta) = \pp(N' < N \leq \beta) + \pp(N < N' \leq \beta) + \pp(N' = N, N \leq \beta).
\]
Since $N'$ and $N$ are independent and identically distributed, $\pp(\max(N', N) \leq \beta) = \pp(N \leq \beta)^2$ and the first two terms on the right of the above expression are identical. Thus, 
\begin{align}
\pp(N' < N \leq \beta) = \frac{1}{2}\lt(\pp(N \leq \beta)^2 - \pp(N' = N, N \leq \beta)\rt).
\label{eqn:nn_theta2}
\end{align}
It is easy to see that 
\[
\pp(N' = N, N \leq \beta) \leq \pp(N' = N) = \exp(- 2\beta) \sum_{n \in \intg_+} \beta^{2n}/(n!)^2,
\]
where the summation term is well-known as the modified Bessel function of {  the first kind of order zero}, denoted as $I_0(2\beta)$. Hence, $\pp(N' = N) = \exp(- 2\beta) I_0(2\beta)$.
From \cite{kasperkovitz1980asymptotic}, we get that $\exp(- 2\beta) I_0(2\beta) \to 0$ as $\beta \to \infty$. We also know from \cite{short2013improved} that 
$\lim_{\beta \to \infty} \pp(N \leq \beta)
= 1/2$.
Therefore, from Eq.~\eqref{eqn:nn_theta2}, it follows that 
$\lim_{\beta \to \infty}\pp(N' < N \leq \beta) = 1/8$.

Now to complete the proof, {  from Eq. \eqref{eqn:log-eff-yhat2}}, it remains to show that 
\[
\lim_{\beta \to \infty} \frac{p_{\lceil \beta \rceil}}{(p^*)^{2-\epsilon}} = \infty.
\]
Since {  the blocking volume added $i$ points is at most $i\upsilon_d$}, from Lemmas~\ref{lem:lower-bound-pn} and \ref{lem:upper-bound-ELX} in Appendix~\ref{app:add-proofs}, we have
\begin{align*}
    p_{\lceil \beta \rceil} 
    &\geq \exp\lt(- \frac{1}{\beta^{\delta}} \sum_{i = 1}^{\lceil\beta \rceil - 1} i \upsilon_d\rt) \exp( - O(\beta^{3 - 2\delta}))\\
    &= \exp\lt(- \frac{\lceil\beta \rceil (\lceil\beta \rceil - 1)\upsilon_d}{2\beta^{\delta}} \rt) \exp( - O(\beta^{3 - 2\delta})).
\end{align*}
From Lemma~\ref{lem:lim-pstar}, we know that $
p^* = \exp\lt(-\frac{ \beta^{2 - \delta} \upsilon_d}{2} (1 + o(1)) \rt)$ for $\delta > 1$. Thus, 
\begin{align}
    \frac{p_{\lceil \beta \rceil}}{(p^*)^{2-\epsilon}}
    &\geq \exp\lt(\frac{\beta^{2 - \delta}\upsilon_d}{2} (2 - \epsilon)(1 + o(1)) - \frac{\lceil\beta \rceil (\lceil\beta \rceil - 1)\upsilon_d}{2\beta^{\delta}} \rt) \exp( - O(\beta^{3 - 2\delta})). \label{eqn:ptheta-by-pstar}
\end{align}
Since $\lceil\beta \rceil (\lceil\beta \rceil - 1) = O(\beta^2)$ and $2 - \delta > 3 - 2 \delta$, by selecting $\epsilon < 1$, we see that the limit of the term on the right-hand side of Eq.~\eqref{eqn:ptheta-by-pstar} tends to infinity as $\beta \to \infty$. 
\end{proof}

\section{Simulation Results}
\label{sec:sim}
In this section, we illustrate the effectiveness of our importance sampling estimator by comparing it to the na\"{i}ve and conditional Monte Carlo estimators through simulation results in {eight} different settings{, covering the edge count, maximum degree, maximum connected component, maximum clique size, number of triangles, and planarity examples}. The Python implementation of our simulation study is available at \url{https://github.com/saratmoka/RareEvents-RandGeoGraphs}.

We recall that $Z_{\tt NMC}$ is the sample mean na\"{i}ve rejection estimator of $\pp(\mX \in A)$, defined by Eq. \eqref{nai.est.def}. That is, $Z_{\tt NMC} = (Y_1+ \cdots + Y_m)/m$, where $Y_1, \dots, Y_m$ are iid copies of $Y$. Similarly, the sample mean conditional Monte Carlo estimator $Z_{\tt CMC}$ and the sample mean importance sampling estimator $Z_{\tt IS}$ are defined by Eq. \eqref{cmc.est.def} and Eq. \eqref{def.wt.zet}, respectively. We use $\mathsf{RV}_{\tt NMC}$, ${\sf RV}_{\tt CMC}$ and ${\sf RV}_{\tt IS}$ to denote the estimated relative variances of $Y$, $\wh Y$ and $\wt Y$, respectively. For instance, 
\[
{\sf RV}_{\tt CMC} = \frac{\frac{1}{m}(\wh Y_1^2+ \cdots + \wh Y_m^2)}{Z^2_{\tt CMC}} -1,
\]
where $\wh Y_1, \dots, \wh Y_m$ are samples used for $Z_{\tt CMC}$. Similarly, we compute ${\sf RV}_{\tt NMC}$ and ${\sf RV}_{\tt IS}$. To ensure that all the sample mean estimators $Z_{\tt NMC}$, $Z_{\tt CMC}$ and $Z_{\tt IS}$ have the same confidence intervals, the number of samples $m$ is selected such that estimates of the relative variances of these sample mean estimators fall below a small fixed value ($0.001$ in our numerical studies). For instance, since ${\sf RV}_{\tt CMC}/m$ is an estimate of the relative variance of the conditional Monte Carlo estimator $Z_{\tt CMC}$, copies of $\wh Y$ are generated until the value of ${\sf RV}_{\tt CMC}/m$ becomes smaller than $0.001$.

\begin{experiment} 
\label{exper-1}
\normalfont
The goal of this numerical study is to estimate the probability of no edges in the Gilbert graph (i.e., hard-spheres model) at different values of the intensity on a fixed window. For that, we take $W = [0,10]^2$ and $\ell = 0$ in Examples {\ref{expl:ec} -- \ref{expl:mcs}} from Section~\ref{sec:prelim}. Table~\ref{tab:relVar_NMC_CMC} presents results for the na\"{i}ve and conditional Monte Carlo estimators for different values of the intensity~$\kappa$. Table~\ref{tab:relVar_ISMC} presents the corresponding results for the importance sampling estimator at different grid sizes.

From these results, we observe that the relative variance (i.e., variance divided by the square of the rare-event probability) of the importance sampling estimator is substantially smaller than that of the other two estimators. For instance, when the intensity is $0.4$, the relative variance of the importance sampling estimator is more than $1000$ times better than that of the conditional Monte Carlo, which is in turn $1000$ times better than the na\"{i}ve Monte Carlo estimator. This means we can make stable and reliable estimates of $\pp(\mX \in A)$ using few samples of $\wt Y$ compared to the other two approaches.

\renewcommand{\thetable}{\arabic{table}a}
\begin{table*}[h]
\caption{Estimated mean and relative variances of the na\"{i}ve Monte Carlo estimator $Y$ and the conditional Monte Carlo estimator $\wh Y$ for Numerical Study~\ref{exper-1}.}
\centering
\label{tab:relVar_NMC_CMC}
\begin{tabular}{@{}lllll@{}}
    \hline
     & \multicolumn{2}{l}{Na\"{i}ve Monte Carlo}
            & \multicolumn{2}{l}{Conditional Monte Carlo}\\
    \cmidrule{1-5}
    $\kappa$    &  $Z_{\tt NMC}$ &   ${\sf RV}_{\tt NMC}$   &   $Z_{\tt CMC}$  &   ${\sf RV}_{\tt CMC}$  \\   
    \hline
$0.1$   &   $0.31$  &   $2.24$  &  $0.31$  &   $0.96$ \\
    \hline
$0.2$   &   $1.81\times 10^{-2}$  &   $54.28$  &   $1.86\times 10^{-2}$  &   $9.91$  \\
    \hline
$0.3$  &   $3.62\times 10^{-4}$    &   $2.76\times 10^{3}$    &  $3.94\times 10^{-4}$     &  $97.1$ \\
    \hline
$0.4$  &   $1.00\times 10^{-6}$    &   $9.99\times 10^{5}$    &  $3.39\times 10^{-6}$     &  $4.16\times 10^{2}$  \\
    \hline
\end{tabular}
\end{table*}
\addtocounter{table}{-1}
\renewcommand{\thetable}{\arabic{table}b}
\begin{table*}[h]
\caption{Estimated mean and relative variances of the importance sampling estimator $\wt Y$ for Numerical Study~\ref{exper-1}, at three different grid sizes.
}
\label{tab:relVar_ISMC}
\centering
\begin{tabular}{@{}lllllll@{}}
\hline
 & \multicolumn{6}{c}{Grid Size}\\
    \cmidrule{2-7}
	   & \multicolumn{2}{l}{$100\times 100$}
            & \multicolumn{2}{l}{$200\times 200$}
                    & \multicolumn{2}{l}{$300 \times 300$}\\
    \hline 
      $\kappa$          &  $Z_{\tt IS}$  &  ${\sf RV}_{\tt IS}$   &   $Z_{\tt IS}$  &  ${\sf RV}_{\tt IS}$  &   $Z_{\tt IS}$  &  ${\sf RV}_{\tt IS}$       \\                   
    \hline
$ 0.1$   &   $0.31$  &   $0.13$  &  $0.31$  &   $0.06$ &   $0.31$  &   $0.04$    \\
    \hline
$0.2$   &   $1.81\times 10^{-2}$  &   $0.54$  &   $1.83\times 10^{-2}$  &   $0.26$  &   $1.83\times 10^{-2}$  &   $0.16$      \\
    \hline
$0.3$  &   $3.57\times 10^{-4}$    &   $1.24$    &  $3.73\times 10^{-4}$     &  $0.55$     &  $3.71\times 10^{-4}$     &  $0.36$        \\
    \hline
$0.4$  &   $3.19\times 10^{-6}$    &   $2.36$    &  $3.52\times 10^{-6}$     &  $0.96$     &  $3.51\times 10^{-6}$     &  $0.63$        \\
    \hline
\end{tabular}
\end{table*}
\renewcommand{\thetable}{\arabic{table}}
\end{experiment}

\begin{experiment}
\label{exper-2}
\normalfont
Here, the goal is to compare the methods for Example~\ref{expl:ec}. Specifically, we estimate the probability that the number of edges in the graph does not exceed a given threshold $\ell$.
For this, we fix $W = [0,20]^2$ and $\kappa = 0.3$ and  vary $\ell$. 
In this setting, the expected number of edges is estimated to be $54$. Since the naïve Monte Carlo estimator requires a prohibitively large number of samples to achieve the same relative error as CMC or IS, we exclude it from this and all following numerical studies. Table~\ref{tab:relVar_ISMC_CMC} compares the conditional Monte Carlo estimator $\widehat{Y}$ and the importance sampling estimator $\widetilde{Y}$. Notably, we observe that the variance of $\widetilde{Y}$ is substantially smaller than that of $\widehat{Y}$ when the rare-event probability is extremely small.
\end{experiment}

\begin{table*}[h]
\caption{ Estimated mean and relative variances of the conditional Monte Carlo estimator $\wh Y$ and the proposed importance sampling estimator $\wt Y$ (with two different grid sizes) for Numerical Study~\ref{exper-2} on the edge count.}
\label{tab:relVar_ISMC_CMC}
\centering
\small
\begin{tabular}{@{}lllllll@{}}
    \hline
	  & \multicolumn{2}{l}{}
            & \multicolumn{4}{c}{Importance Sampling}\\
    \cmidrule{4-7}
	  & \multicolumn{2}{l}{Conditional Monte Carlo}
            & \multicolumn{2}{l}{Grid size: $200 \times 200$}
                    & \multicolumn{2}{l}{Grid size: $300 \times 300$}\\
    \hline 
     $\ell$          &  $Z_{\tt CMC}$  &  ${\sf RV}_{\tt CMC}$   &   $Z_{\tt IS}$  &  ${\sf RV}_{\tt IS}$  &   $Z_{\tt IS}$  &  ${\sf RV}_{\tt IS}$       \\  
    \hline
$5$   &   $5.00\times 10^{-10}$  &   $4.81\times 10^{4}$  &  $4.35\times 10^{-10}$  &   $3.41\times 10^{3}$ &   $4.31\times 10^{-10}$  &   $2.96\times 10^{3}$    \\
    \hline
$10$   &   $2.56\times 10^{-7}$  &   $2.14\times 10^{3}$  &   $2.32\times 10^{-7}$  &   $5.08\times 10^{2}$  &   $1.68\times 10^{-7}$  &   $2.00\times 10^{2}$      \\
    \hline
$15$  &   $1.51\times 10^{-5}$  &   $1.17\times 10^{2}$  &   $1.60\times 10^{-5}$  &   $9.59\times 10^{1}$  &   $1.61\times 10^{-5}$  &   $1.06\times 10^{2}$      \\
    \hline
\end{tabular}
\end{table*}

\begin{experiment}
\label{exper-3}
\normalfont
Our next objective is to compare the methods for Example~\ref{expl:deg} by estimating the probability that the maximum degree of the graph does not exceed a given threshold $\ell$. For this, we set $W = [0, 10]^2$ and vary both the intensity $\kappa$ and the threshold $\ell$. For our importance sampling, we fix the grid size to be $200 \times 200$. Here, we again observe that the relative variance of the importance sampling estimator $\wt Y$ is significantly smaller than that of the conditional Monte Carlo estimator $\wh Y$, see Table~\ref{tab:relVar_exper_3}.

\begin{table*}[h]
\caption{Estimated mean and relative variances of the conditional Monte Carlo estimator $\wh Y$ and the proposed importance sampling estimator $\wt Y$ for Numerical Study~\ref{exper-3} on the maximum degree.} 
\label{tab:relVar_exper_3}
\centering
\begin{tabular}{@{}lllllr@{}}
    \hline
	  &  & \multicolumn{2}{l}{Conditional Monte Carlo}
            & \multicolumn{2}{l}{Importance Sampling}
                    \\
    \hline 
     $\kappa$      & $\ell$    &  $Z_{\tt CMC}$  &  ${\sf RV}_{\tt CMC}$   &   $Z_{\tt IS}$  &  ${\sf RV}_{\tt IS}$ \\  
    \hline
$1$   & $4$ &  $3.34 \times 10^{-3}$ & $47.89$ &   $3.64 \times 10^{-3}$ & $4.54$        \\
    \hline
$1.5$   & $5$ &  $1.22 \times 10^{-5}$ & $6.38 \times 10^{2}$ &   $1.09 \times 10^{-5}$ & $54.09$      \\
    \hline
$2$  & $6$ &   $2.19 \times 10^{-8}$ & $1.19 \times 10^{5}$ &   $2.08 \times 10^{-8}$ & $2.17 \times 10^{2}$         \\
    \hline
\end{tabular}
\end{table*}
\end{experiment}

\begin{experiment}
 
\label{exper-4}
\normalfont
In this numerical study, we compare the methods for Example~\ref{expl:mcc}. Specifically, we estimate the probability that the maximum connected component of the graph does not exceed $\ell + 1$. We fix $W = [0, 10]^2$, $\ell = 1$, and vary the intensity $\kappa$, with the grid size fixed at $200 \times 200$. We observe that the relative variance of the importance sampling estimator $\wt{Y}$ is substantially smaller than that of the conditional Monte Carlo estimator $\wh{Y}$, as shown in Table~\ref{tab:relVar_exper_4}.
\end{experiment}

\begin{table*}[h]
 
\caption{Estimated mean and relative variances of the conditional Monte Carlo estimator $\wh Y$ and the proposed importance sampling estimator $\wt Y$ for Numerical Study~\ref{exper-4} on maximum connected components.}
\label{tab:relVar_exper_4}
\centering
\begin{tabular}{@{}lllll@{}}
    \hline
      & \multicolumn{2}{l}{Conditional Monte Carlo}
            & \multicolumn{2}{l}{Importance Sampling} \\
    \hline
     $\kappa$  &  $Z_{\tt CMC}$  &  ${\sf RV}_{\tt CMC}$   &   $Z_{\tt IS}$  &  ${\sf RV}_{\tt IS}$ \\
    \hline
$0.3$  &  $5.78 \times 10^{-2}$  &  $5.26$  &  $5.81 \times 10^{-2}$  &  $1.27$  \\
    \hline
$0.4$  &  $4.25 \times 10^{-3}$  &  $2.72\times 10^{1}$  &  $4.55 \times 10^{-3}$  &  $3.00$  \\
    \hline
$0.5$  &  $2.01 \times 10^{-4}$  &  $1.15 \times 10^{2}$  &  $2.33 \times 10^{-4}$  &  $6.70$  \\
    \hline
\end{tabular}
\end{table*}

\begin{experiment}
 
\label{exper-5}
\normalfont
In this numerical study, we compare the methods for Example~\ref{expl:mcs}. Specifically, we estimate the probability that the maximum clique size of the graph does not exceed $\ell + 1$. We fix $W = [0, 10]^2$, $\ell = 1$ (i.e., the graph is triangle-free), and vary the intensity $\kappa$, with the grid size fixed at $200 \times 200$. Table~\ref{tab:relVar_exper_5} shows that the importance sampling estimator $\wt Y$ again achieves substantially smaller relative variance than the conditional Monte Carlo estimator $\wh Y$. 
\end{experiment}

\begin{table*}[h]
 
\caption{Estimated mean and relative variances of the conditional Monte Carlo estimator $\wh Y$ and the proposed importance sampling estimator $\wt Y$ for Numerical Study~\ref{exper-5} on maximum clique sizes.}
\label{tab:relVar_exper_5}
\centering
\begin{tabular}{@{}lllll@{}}
    \hline
      & \multicolumn{2}{l}{Conditional Monte Carlo}
            & \multicolumn{2}{l}{Importance Sampling} \\
    \hline
     $\kappa$  &  $Z_{\tt CMC}$  &  ${\sf RV}_{\tt CMC}$   &   $Z_{\tt IS}$  &  ${\sf RV}_{\tt IS}$ \\
    \hline
$0.4$  &  $5.46 \times 10^{-2}$  &  $5.84$  &  $6.24 \times 10^{-2}$  &  $0.62$  \\
    \hline
$0.5$  &  $6.08 \times 10^{-3}$  &  $2.01 \times 10^{1}$  &  $8.75 \times 10^{-3}$  &  $1.15$  \\
    \hline
$0.6$  &  $7.60 \times 10^{-4}$  &  $1.05 \times 10^{2}$  &  $7.87 \times 10^{-4}$  &  $1.87$  \\
    \hline
\end{tabular}
\end{table*}

\begin{experiment}
 
\label{exper-6}
\normalfont
In this numerical study, we compare the methods for Example~\ref{expl:ntg}. Specifically, we estimate the probability that the graph contains no triangles (i.e., $\ell = 0$). We fix $W = [0, 10]^2$ and vary the intensity $\kappa$, with the grid size fixed at $200 \times 200$. As shown in Table~\ref{tab:relVar_exper_6}, the importance sampling estimator $\wt Y$ achieves substantially smaller relative variance than the conditional Monte Carlo estimator $\wh Y$. Note that for $\ell = 1$ in Example~\ref{expl:mcs} and $\ell = 0$ in Example~\ref{expl:ntg}, both rare events correspond to the graph being triangle-free, and the estimates from Numerical Studies~\ref{exper-5} and~\ref{exper-6} are in close agreement.
\end{experiment}

\begin{table*}[h]
 
\caption{Estimated mean and relative variances of the conditional Monte Carlo estimator $\wh Y$ and the proposed importance sampling estimator $\wt Y$ for Numerical Study~\ref{exper-6} on seeing no triangles in the graph.}
\label{tab:relVar_exper_6}
\centering
\begin{tabular}{@{}lllll@{}}
    \hline
      & \multicolumn{2}{l}{Conditional Monte Carlo}
            & \multicolumn{2}{l}{Importance Sampling} \\
    \hline
     $\kappa$  &  $Z_{\tt CMC}$  &  ${\sf RV}_{\tt CMC}$   &   $Z_{\tt IS}$  &  ${\sf RV}_{\tt IS}$ \\
    \hline
$0.4$  &  $5.46 \times 10^{-2}$  &  $5.84$  &  $6.24 \times 10^{-2}$  &  $0.62$  \\
    \hline
$0.5$  &  $6.08 \times 10^{-3}$  &  $2.01 \times 10^{1}$  &  $8.75 \times 10^{-3}$  &  $1.15$  \\
    \hline
$0.6$  &  $7.60 \times 10^{-4}$  &  $1.05 \times 10^{2}$  &  $7.87 \times 10^{-4}$  &  $1.87$  \\
    \hline
\end{tabular}
\end{table*}

\begin{experiment}
 
\label{exper-7}
\normalfont
In this numerical study, we compare the methods for Example~\ref{expl:planar}. Specifically, we estimate the probability that the Gilbert graph is planar. We fix $W = [0, 10]^2$ and vary the intensity $\kappa$, with the grid size fixed at $100 \times 100$. The blocking rule uses a $K_5$- and $K_{3,3}$-based incremental criterion (as described in Section~\ref{sec:is}), while the stopping condition uses a full planarity check. As shown in Table~\ref{tab:relVar_exper_7}, the importance sampling estimator $\wt Y$ achieves substantially smaller relative variance than the conditional Monte Carlo estimator $\wh Y$.
\end{experiment}

\begin{table*}[h]
 
\caption{Estimated mean and relative variances of the conditional Monte Carlo estimator $\wh Y$ and the proposed importance sampling estimator $\wt Y$ for Numerical Study~\ref{exper-7} on estimating the probability of graph being non-planar.}
\label{tab:relVar_exper_7}
\centering
\begin{tabular}{@{}lllll@{}}
    \hline
      & \multicolumn{2}{l}{Conditional Monte Carlo}
            & \multicolumn{2}{l}{Importance Sampling} \\
    \hline
     $\kappa$  &  $Z_{\tt CMC}$  &  ${\sf RV}_{\tt CMC}$   &   $Z_{\tt IS}$  &  ${\sf RV}_{\tt IS}$ \\
    \hline
$1.1$  &  $4.58 \times 10^{-2}$  &  $7.31$  &  $4.12 \times 10^{-2}$  &  $2.08$  \\
    \hline
$1.2$  &  $1.51 \times 10^{-2}$  &  $1.58 \times 10^{1}$  &  $1.30 \times 10^{-2}$  &  $4.01$  \\
    \hline
$1.3$  &  $3.87 \times 10^{-3}$  &  $4.25 \times 10^{1}$  &  $3.37 \times 10^{-3}$  &  $7.27$  \\
    \hline
\end{tabular}
\end{table*}

\begin{experiment}
 
\label{exper-8}
\normalfont
We compare the computational cost of CMC and IS across all six examples at the precision target ${\sf RV}/m < \varepsilon = 0.01$, with $W = [0,10]^2$, $r = 1$, and a $100 \times 100$ IS grid. For the five threshold examples we set $\kappa = 0.65$ at levels $\ell$ yielding $Z \approx 10^{-4}$ (for MD, $\ell = 2$ is the closest feasible choice, giving $Z \approx 9.4 \times 10^{-4}$); for Planarity we use $\kappa = 1.5$. Since NMC is impractical, we compare CMC and IS. As seen in Table~\ref{tab:timing_exper_8}, $m_{\tt IS}/m_{\tt CMC} \approx {\sf RV}_{\tt IS}/{\sf RV}_{\tt CMC}$, consistent with the stopping criterion, and IS achieves wall-clock speedups of $36$--$230\times$ over CMC for the five threshold examples. For Planarity, CMC is infeasible at $Z \approx 10^{-4}$ (extrapolated cost $\approx 498$~s from a 3000-sample pilot), whereas IS converges in $66$~s. The product $\text{Time} \times {\sf RV}$, reported in the last two columns of Table~\ref{tab:timing_exper_8}, is two to four orders of magnitude smaller for IS than for CMC across all examples. Table~\ref{tab:timing_w20} repeats this comparison at the same window $W = [0,10]^2$ with $\kappa = 0.8$ for the threshold examples and $\kappa = 1.7$ for Planarity, yielding $Z \approx 10^{-6}$, one order of magnitude smaller than in Table~\ref{tab:timing_exper_8}. The IS speedups are even more pronounced at this lower probability level, with $\text{Time} \times {\sf RV}$ ratios ranging from $12\times$ to over $1500\times$ smaller for IS than for CMC across the five threshold examples. For Planarity, CMC is again infeasible at $Z \approx 10^{-6}$ (extrapolated cost $\approx 385$~s from a 3000-sample pilot), and IS requires similar wall-clock time ($431$~s) but with an RV five times smaller, resulting in a $4\times$ improvement in $\text{Time} \times {\sf RV}$. All running times were measured on a Macbook Pro with an Apple M4 Pro chip and 24GB of unified memory using a single-core. 
\end{experiment}

\renewcommand{\thetable}{\arabic{table}a}
\begin{table*}[h]
 
\caption{Computational cost comparison of CMC and IS across all six examples at the same precision target (${\sf RV}/m < 0.01$). Settings: $W = [0,10]^2$, $r = 1$, IS grid $100 \times 100$. For threshold examples $\kappa = 0.65$; for Planarity $\kappa = 1.5$. The symbol $\ell$ is as defined in each respective example (not applicable for Planarity). $^\dagger$CMC $m$ and time are extrapolated from a 3000-iteration pilot.}
\label{tab:timing_exper_8}
\centering
\resizebox{\textwidth}{!}{%
\begin{tabular}{@{}lllll llll llll@{}}
    \hline
    & & & & \multicolumn{4}{l}{Conditional MC} & \multicolumn{4}{l}{Importance Sampling} \\
    \hline
    Example & $\kappa$ & $\ell$ & $Z_{\tt CMC}$ & ${\sf RV}_{\tt CMC}$ & $m_{\tt CMC}$ & Time$_{\tt CMC}$ (s) & T$\times$RV$_{\tt CMC}$ & ${\sf RV}_{\tt IS}$ & $m_{\tt IS}$ & Time$_{\tt IS}$ (s) & T$\times$RV$_{\tt IS}$ \\
    \hline
EC (Ex.~\ref{expl:ec})  &  $0.65$  &  $15$  &  $1.15 \times 10^{-4}$  &  $17.6$  &  $10099$  &  $1.4$  &  $24.6$  &  $9.11$  &  $1099$  &  $0.02$  &  $0.18$  \\
    \hline
MD (Ex.~\ref{expl:deg}) &  $0.65$  &  $2$   &  $9.38 \times 10^{-4}$  &  $51.4$  &  $5320$   &  $0.7$  &  $36.0$  &  $6.82$  &  $783$   &  $0.02$  &  $0.13$  \\
    \hline
MCC (Ex.~\ref{expl:mcc}) &  $0.65$  &  $2$   &  $2.48 \times 10^{-4}$  &  $211$   &  $21208$  &  $2.4$  &  $506$   &  $19.9$  &  $2140$  &  $0.02$  &  $0.31$  \\
    \hline
NTG (Ex.~\ref{expl:ntg}) &  $0.65$  &  $0$   &  $1.91 \times 10^{-4}$  &  $189$   &  $19029$  &  $2.0$  &  $378$   &  $5.65$  &  $690$   &  $0.01$  &  $0.05$ \\
    \hline
MCS (Ex.~\ref{expl:mcs}) &  $0.65$  &  $1$   &  $1.91 \times 10^{-4}$  &  $189$   &  $19029$  &  $2.0$  &  $378$   &  $5.65$  &  $690$   &  $0.01$  &  $0.05$ \\
    \hline
Planarity (Ex.~\ref{expl:planar}) &  $1.5$  &  $-$   &  $1.42 \times 10^{-4}$  &  $192$   &  $19241^\dagger$  &  $498^\dagger$  &  $9.6 \times 10^4{}^\dagger$  &  $11.4$  &  $1292$  &  $66.00$  &  $752.00$  \\
    \hline
\end{tabular}
}
\end{table*}
\addtocounter{table}{-1}
\renewcommand{\thetable}{\arabic{table}b}
\begin{table*}[h]
 
\caption{Computational cost comparison of CMC and IS across all six examples at the same precision target (${\sf RV}/m < 0.01$). Settings: $W = [0,10]^2$, $r = 1$, IS grid $100 \times 100$. For threshold examples $\kappa = 0.8$ ($Z \approx 10^{-6}$); for Planarity $\kappa = 1.7$. The symbol $\ell$ is as defined in each respective example (not applicable for Planarity). $^\dagger$CMC $m$ and time are extrapolated from a 3000-iteration pilot.}
\label{tab:timing_w20}
\centering
\resizebox{\textwidth}{!}{%
\begin{tabular}{@{}lllll llll llll@{}}
    \hline
    & & & & \multicolumn{4}{l}{Conditional MC} & \multicolumn{4}{l}{Importance Sampling} \\
    \hline
    Example & $\kappa$ & $\ell$ & $Z_{\tt CMC}$ & ${\sf RV}_{\tt CMC}$ & $m_{\tt CMC}$ & Time$_{\tt CMC}$ (s) & T$\times$RV$_{\tt CMC}$ & ${\sf RV}_{\tt IS}$ & $m_{\tt IS}$ & Time$_{\tt IS}$ (s) & T$\times$RV$_{\tt IS}$ \\
    \hline
EC (Ex.~\ref{expl:ec})  &  $0.8$  &  $15$  &  $1.54 \times 10^{-7}$  &  $119$   &  $12075$   &  $1.7$   &  $202$               &  $20.8$  &  $2174$   &  $0.79$  &  $16.4$  \\
    \hline
MD (Ex.~\ref{expl:deg}) &  $0.8$  &  $2$   &  $1.07 \times 10^{-5}$  &  $508$   &  $50853$   &  $6.6$   &  $3328$              &  $15.2$  &  $1673$   &  $0.97$  &  $14.8$  \\
    \hline
MCC (Ex.~\ref{expl:mcc}) &  $0.8$  &  $2$  &  $1.27 \times 10^{-6}$  &  $1357$  &  $135858$  &  $15.7$  &  $2.1 \times 10^4$   &  $125$   &  $12595$  &  $0.89$  &  $112$   \\
    \hline
NTG (Ex.~\ref{expl:ntg}) &  $0.8$  &  $0$  &  $1.79 \times 10^{-6}$  &  $1396$  &  $139986$  &  $15.5$  &  $2.2 \times 10^4$   &  $15.9$  &  $1699$   &  $0.91$  &  $14.4$  \\
    \hline
MCS (Ex.~\ref{expl:mcs}) &  $0.8$  &  $1$  &  $1.79 \times 10^{-6}$  &  $1396$  &  $139986$  &  $15.0$  &  $2.1 \times 10^4$   &  $15.9$  &  $1699$   &  $1.06$  &  $16.9$  \\
    \hline
Planarity (Ex.~\ref{expl:planar}) &  $1.7$  &  $-$  &  $5.78 \times 10^{-7}$  &  $407$  &  $40670^\dagger$  &  $385^\dagger$  &  $1.6 \times 10^5{}^\dagger$  &  $84.3$  &  $8640$  &  $431$  &  $3.6 \times 10^4$  \\
    \hline
\end{tabular}
}
\end{table*}
\renewcommand{\thetable}{\arabic{table}}

\section{Conclusion}
\label{sec:conclusion}
In this paper, we considered the problem of estimating rare-event probabilities for random geometric graphs, also known as Gilbert graphs. We proposed an easily implementable and efficient importance sampling method for rare-event estimation. Using analysis and simulations, we compared its performance with existing methods: na\"{i}ve and conditional Monte Carlo estimators.  In particular, we showed that the importance sampling estimator always exhibits smaller variance than the other two methods. Furthermore, we established an asymptotic regime where the importance sampling estimator is efficient while the other estimators are inefficient. Our simulation results show that the proposed estimator can have variance thousands of times smaller than the other estimators when the rare-event probabilities are extremely small. 
\medskip

\section*{Declarations}
\noindent {\bf Author Contributions.} SM was primarily responsible for the conception, mathematical proofs, and simulations. CH made substantial contributions by writing the introduction and providing comprehensive literature review and relevant resources. All authors contributed to polishing the structure and presentation of the material.
\medskip

\noindent {\bf Funding}. CH was supported by a research grant (VIL69126) from VILLUM FONDEN.
\medskip

\noindent {\bf Code Availability.} The Python implementation of our numerical methods is available at \url{https://github.com/saratmoka/RareEvents-RandGeoGraphs}. We also released these methods as package available at \url{https://pypi.org/project/pyregg}.
\medskip

\noindent {\bf Conflict of Interest}
Christian Hirsch is a member of the editorial board for Methodology and Computing in Applied Probability.

\begin{appendices}
\section{Proof of Lemma \ref{lem:upper-bound}}\label{app:add-proofs}
We now provide a proof of Lemma~\ref{lem:upper-bound}. For this we use two lemmas stated below: Lemma~\ref{lem:lower-bound-pn} and Lemma~\ref{lem:upper-bound-ELX}. Note that the blocking volumes are monotonically non-decreasing. That is, $B(\mX_{i-1}) \subseteq B(\mX_{i})$ is for each $i$.  Also, the blocking volume added by a point is at most $\upsilon_d$, the volume of a unit radius sphere. In other words, with 
\[
V_i = \vol(B(\mX_i)),
\]
we have $V_i - V_{i-1} \leq \upsilon_d$,
where the equality holds when the center of the $i$-th sphere is one unit away from the boundary of the window and $2$ units away from the centers of the other spheres.
As a consequence, we get that
\begin{align}
\label{eqn:bounds_on_vi}
V_{i}  = \sum_{j = 1}^{i} \lt(V_{j} - V_{j-1}\rt) \leq i \upsilon_d.
\end{align}
Throughout the section, we assume that $\lambda = \beta^{\delta/d}$ with $1 < \delta  < 2$.
\begin{lemma}
    \label{lem:lower-bound-pn}
    For all sufficiently large values of $\beta>0$, it holds that
\begin{align*}
    p_n &\geq \exp\lt( - \frac{1}{\beta^{\delta}}\sum_{i = 1}^{n -1} \ee_\mu[V_i]\rt) \exp\lt( - \sum_{i = 1}^{n -1} \sum_{j = 2}^\infty \frac{(i\upsilon_d)^j}{\beta^{j\delta}}\rt), \quad n \leq 2\beta.
\end{align*}
\end{lemma}

\begin{proof}
Recall from Eq.~\eqref{eqn:radon-nikodym-der} that 
\[
L(\mX_n) = \prod_{i=1}^{n-1}\Big(1 - \frac{ V_{i}}{\lambda^d}\Big)^+ = \prod_{i=1}^{n-1}\Big(1 - \frac{ V_{i}}{\beta^\delta}\Big)^+.
\]
From Eq.~\eqref{eqn:bounds_on_vi}, for all $i < n \leq 2\beta$, we have 
\[
V_i \leq i\upsilon_d < n \upsilon_d \leq 2\beta \upsilon_d.
\]
Since $\delta > 1$, for sufficiently large values of $\beta$, we have 
$2\beta \upsilon_d < \beta^{\delta}$, and thus, $V_i/\beta^\delta < 1$ for all $i < n \leq 2\beta$.
Consequently, using the Taylor expansion of $\log(1 - u)$, $u < 1$, and the definition of $p_n$, we obtain
\begin{align*}
    p_n = \ee_\mu \lt[\exp\lt( \sum_{i = 1}^{n -1} \log(1 - V_i/\beta^{\delta})\rt) \rt]
        = \ee_\mu \lt[\exp\lt( - \sum_{i = 1}^{n -1} \sum_{j = 1}^\infty \frac{V_i^j}{\beta^{j\delta}}\rt) \rt],
\end{align*}
for large values of $\beta$ with $n \leq 2\beta$. Furthermore, we have
\begin{align*}
    p_n = \ee_\mu \lt[\exp\lt( - \frac{1}{\beta^{\delta}}\sum_{i = 1}^{n -1} V_i\rt) \exp\lt( - \sum_{i = 1}^{n -1} \sum_{j = 2}^\infty \frac{V_i^j}{\beta^{j\delta}}\rt) \rt],
\end{align*}
Then, from Eq.~\eqref{eqn:bounds_on_vi}, 
\begin{align*}
    p_n &\geq \ee_\mu \lt[\exp\lt( - \frac{1}{\beta^{\delta}}\sum_{i = 1}^{n -1} V_i\rt) \rt] \exp\lt( - \sum_{i = 1}^{n -1} \sum_{j = 2}^\infty \frac{(i\upsilon_d)^j}{\beta^{j\delta}}\rt)\\
        &\geq \exp\lt( - \frac{1}{\beta^{\delta}}\sum_{i = 1}^{n -1} \ee_\mu[V_i]\rt) \exp\lt( - \sum_{i = 1}^{n -1} \sum_{j = 2}^\infty \frac{(i\upsilon_d)^j}{\beta^{j\delta}}\rt),
\end{align*}
where the last inequality 
is a consequence of Jensen's inequality since $\exp(\cdot)$ is convex.
\end{proof} 

\begin{lemma}
    \label{lem:upper-bound-ELX}
    For any $n \leq 2\beta$, 
    \begin{align}
    \sqrt{\ee_\mu\lt[\exp\lt(\frac{2}{\beta^{\delta}}  \sum_{i =1}^{n-1} (i\upsilon_d - V_i)\rt)\rt]} = \exp\lt( O\lt(\beta^{2 -{\delta(d+1)/d}} + \beta^{3 - 2\delta} \rt)\rt),
    \label{eqn:part1-of-lemma}
\end{align}
and 
\begin{align}
 \sum_{i = 1}^{n -1} \sum_{j = 2}^\infty\frac{(i\upsilon_d)^j}{\beta^{j\delta}} = O( \beta^{3 - 2\delta}). 
\label{eqn:part2-of-lemma}
\end{align}
\end{lemma}
\begin{proof}
For each $i = 1, \dots, n-1$, define the indicator random variable
$I_i = \ind( V_{i} - V_{i-1} < \upsilon_d)$.
Then, 
\begin{align}
i\upsilon_d - V_{i} &= \sum_{j = 1}^{i} \lt(\upsilon_d - (V_{j} - V_{j-1}) \rt) =  \sum_{j =1}^i I_i\lt(\upsilon_d - (V_{j} - V_{j-1}) \rt) 
\leq {  \upsilon_d \sum_{j =1}^i I_j,}
\label{eqn:low_bdd_Bi}
\end{align}
Thus, 
\begin{align*}
\sum_{i =1}^{n-1} (i\upsilon_d - V_i) \leq \upsilon_d \sum_{i = 1}^{n-1} \sum_{j = 1}^i I_j = \upsilon_d \sum_{j = 1}^{n-1} (n - j) I_j,
\end{align*}
for each $i = 1,\dots, n-1$.
Note that $I_i = 1$ means the $i^{th}$ point $X_i$ is generated within $1$ unit from the boundary of the window~$W$ or within $2$ units from the centers of all the existing $i-1$ points. {  Observe that  for any $j =1, \dots, i-1$, the probability of $\|X_i - X_j\|_2 \leq 2$ is at most $2^d \upsilon_d/(\beta^{\delta} - V_i)$. This is because, $X_j$ is generated on non-blocking region created by the points $1, \dots, i-1$ and $x_i$ has to be within two units distance from $X_i$.} Thus, we have
\[
\pp_\mu \lt( I_i = 1 | X_1,\dots, X_{i-1}\rt) \leq 
\frac{\beta^{\delta} - (\beta^{\delta/d} - 2)^d}{\beta^{\delta} - V_i} + \sum_{j = 1}^{i-1} \frac{2^d \upsilon_d}{\beta^{\delta} - V_i}.
\]
Since $V_i \leq i \upsilon_d,$ with $r_i =  \frac{\beta^{\delta} - (\beta^{{\delta/d} } -2)^d}{\beta^{\delta} - i\upsilon_d} + \frac{i \upsilon_d}{\beta^{\delta} - i \upsilon_d}$, we have 
$\pp_\mu\lt( I_i = 1| X_1,\dots, X_{i-1}\rt) \leq r_i.$
Consequently,
\begin{align*}
\ee_\mu\lt[\exp\lt(\frac{2}{\beta^{\delta}}  \sum_{i =1}^{n-1} (i\upsilon_d - V_i)\rt)\rt] &\leq \ee_\mu\lt[\exp\lt(\frac{2 \upsilon_d}{\beta^{\delta}} \sum_{j = 1}^{n-1} (n - j) I_j \rt)\rt].
\end{align*}
Observe that the right side of the above inequality is equal to 
\begin{align*}
{  \ee_\mu\lt[\exp\lt(\frac{2 \upsilon_d}{\beta^{\delta}} \sum_{j = 1}^{n-2} (n - j) I_j \rt) \ee_\mu\lt[ \exp\lt(\frac{2 \upsilon_d}{\beta^{\delta}} I_{n-1} \rt)\big| X_1,\dots, X_{n-2}\rt]\rt]},
\end{align*}
which can be upper bounded by
\begin{align*}
\ee_\mu\lt[\exp\lt(\frac{2 \upsilon_d}{\beta^{\delta}} \sum_{j = 1}^{n-2} (n - j) I_j \rt) \rt] \lt((1 - r_{n-1}) + r_{n-1} \exp(2 \upsilon_d/\beta^{\delta}) \rt).
\end{align*}
By repeating the same procedure for all the terms in the above expectation and substituting the values of $c_i$'s, we establish that 
\begin{align}
\ee_\mu\lt[\exp\lt(\frac{2}{\beta^{\delta}}  \sum_{i =1}^{n-1} (i\upsilon_d - V_i)\rt)\rt] 
&\leq \prod_{i = 1}^{n-1}\lt(1 + r_{n-i} (e^{2i \upsilon_d/\beta^{\delta}} -1) \rt)\nonumber\\
&\leq \exp\lt( \sum_{i = 1}^{n-1} r_{n-i} (e^{2i \upsilon_d/\beta^{\delta}} -1) \rt). 
\label{eqn:exp-exp-bdd}
\end{align}
Then, notice that since $\delta > 1$, for each $i < n \leq 2\beta$, $iv/\beta^{\delta} \to 0$ as $\beta \to \infty$. Furthermore, 
 $\beta^{\delta} - (\beta^{\delta/d} - 2)^d = O(\beta^{\delta(d -1)/d})$, and hence,  
$r_i = O\lt((\beta^{\delta(d-1)/d} + i)/\beta^{\delta}\rt)$.
Also,
\begin{align*}
    \exp(2i \upsilon_d/\beta^{\delta}) -1 = \sum_{j=1}^\infty \frac{(2i \upsilon_d)^j}{j!\beta^{j\delta}} = \frac{2i\upsilon_d}{\beta^{\delta}} \sum_{j=0}^\infty \frac{(2i \upsilon_d)^j}{(j+1)!\beta^{j\delta}} \leq \frac{2i\upsilon_d}{\beta^{\delta}} \exp(2i \upsilon_d/\beta^{\delta}).
\end{align*}
Since $\exp(2i \upsilon_d/\beta^{\delta}) \to 1$ as $\beta \to \infty$ for all $i < n \leq 2\beta$, we have
\[
\sum_{i = 1}^{n-1} r_{n-i} (\exp(2i \upsilon_d/\beta^{\delta}) -1) = O\lt(\frac{n^2\beta^{\delta(d-1)/d}}{\beta^{2\delta}} + \frac{n^3}{\beta^{2\delta}} \rt) = O\lt(\beta^{2 - {\delta(d+1)/d}} + \beta^{3 - 2\delta} \rt),
\]
which completes the proof of \eqref{eqn:part1-of-lemma} using \eqref{eqn:exp-exp-bdd}.
Now to prove \eqref{eqn:part1-of-lemma}, observe for large $\beta$ that if $i < n \leq 2\beta$, 
\begin{align*}
    \sum_{i = 1}^{n -1} \sum_{j = 2}^\infty \frac{(i\upsilon_d)^j}{\beta^{j\delta}} 
    = \frac{\upsilon_d^2}{\beta^{2\delta}} \sum_{i = 1}^{n -1} i^2 \sum_{j = 0}^\infty \frac{(i\upsilon_d)^j}{\beta^{j\delta}} 
    = \frac{\upsilon_d^2}{\beta^{2\delta}} \sum_{i = 1}^{n -1} i^2 \frac{1}{1 - i\upsilon_d/\beta^{\delta}}
    \leq \frac{\upsilon_d^2}{\beta^{2\delta}} \frac{1}{1 - n\upsilon_d/\beta^{\delta}}\sum_{i = 1}^{n -1} i^2,
\end{align*}
where we used $i\upsilon_d < n\upsilon_d$ in the last inequality.
We complete the proof by observing that $n/\beta^\delta \leq 2\beta/\beta^\delta \to 0$ as $\beta \to \infty$ and 
$\sum_{i < n} i^2 \leq \sum_{i < 2\beta} i^2 = O(\beta^3).$
\end{proof}
\begin{proof}[Proof of Lemma~\ref{lem:upper-bound}]
Observe that
\begin{align*}
    \ee_\mu[L(\mX_n)^2] &\leq \ee_\mu\lt[\exp\lt(-  \frac{2}{\beta^{\delta}}\sum_{i =1}^{n-1} V_i\rt)\rt]\\
                        &= \exp\lt(- \frac{2}{\beta^{\delta}} \sum_{i =1}^{n-1} \ee_\mu[V_i]\rt) \ee_\mu\lt[\exp\lt(-  \frac{2}{\beta^{\delta}}  \sum_{i =1}^{n-1} (V_i - \ee_{\mu}[V_i])\rt)\rt]\\
                        &\leq \exp\lt(-  \frac{2}{\beta^{\delta}} \sum_{i =1}^{n-1} \ee_\mu[V_i]\rt) \ee_\mu\lt[\exp\lt(\frac{2}{\beta^{\delta}}  \sum_{i =1}^{n-1} (i\upsilon_d - V_i)\rt)\rt],
\end{align*}
where the last inequality follows from the upper bound in Eq.~\eqref{eqn:bounds_on_vi}. Thus, using Lemma~\ref{lem:lower-bound-pn}, we have \begin{align}
\frac{\sqrt{\ee_\mu[L(\mX_n)^2]}}{p_n} \leq \exp\lt( \sum_{i = 1}^{n -1} \sum_{j = 2}^\infty \frac{(i\upsilon_d)^j}{\beta^{j\delta}}\rt) \sqrt{\ee_\mu\lt[\exp\lt(\frac{2}{\beta^{\delta}}  \sum_{i =1}^{n-1} (i\upsilon_d - V_i)\rt)\rt]}.\label{eqn:upper-bdd-1}
\end{align}
We now complete the proof using Lemma~\ref{lem:upper-bound-ELX}.
\end{proof}

\end{appendices}

\bibliography{Ref}

@article{doge2004grand,
  title={Grand canonical simulations of hard-disk systems by simulated tempering},
  author={D{\"o}ge, Gunter and Mecke, Klaus and M{\o}ller, Jesper and Stoyan, Dietrich and Waagepetersen, Rasmus P},
  journal={International Journal of Modern Physics C},
  volume={15},
  number={01},
  pages={129--147},
  year={2004},
  publisher={World Scientific}
}

@INPROCEEDINGS{kenniche2010random,
  author={Kenniche, H. and Ravelomananana, V.},
  booktitle={Proceedings of the 2nd International Conference on Computer and Automation Engineering (ICCAE)}, 
  title={Random geometric graphs as model of wireless sensor networks}, 
  year={2010},
  volume={4},
  number={},
  pages={103-107}
}

@article{kasperkovitz1980asymptotic,
  title={Asymptotic approximations for modified {B}essel functions},
  author={Kasperkovitz, P},
  journal={J. Math. Phys.},
  volume={21},
  number={1},
  pages={6--13},
  year={1980},
  publisher={AIP Publishing}
}

@article{short2013improved,
  title={Improved inequalities for the Poisson and binomial distribution and upper tail quantile functions},
  author={Short, Michael},
  journal={International Scholarly Research Notices},
  volume={2013},
  number={1},
  pages={412958},
  year={2013},
  publisher={Wiley Online Library}
}

@book{krauth2006statistical,
    AUTHOR = {Krauth, W.},
     TITLE = {Statistical Mechanics},
  YEAR = {2006},
 PUBLISHER = {Oxford University Press},
Address = {Oxford},
}

@article {HMTK20,
    AUTHOR = {Hirsch, C. and Moka, S. B. and Taimre, T. and
              Kroese, D. P.},
     TITLE = {Rare events in random geometric graphs},
   JOURNAL = {Methodol. Comput. Appl. Probab.},
  FJOURNAL = {Methodology and Computing in Applied Probability},
    VOLUME = {24},
      YEAR = {2022},
    NUMBER = {3},
     PAGES = {1367--1383},
}

@book {RK2017,
    AUTHOR = {Rubinstein, R.Y. and Kroese, D. P.},
     TITLE = {Simulation and the {M}onte {C}arlo Method},
      EDITION = {3rd},
 PUBLISHER = {J. Wiley \& Sons},
        Address = {Chichester},
      YEAR = {2017}
}

@article {MJM17,
    AUTHOR = {Moka, S. and Juneja, S. and Mandjes, M.},
     TITLE = {Rejection- and importance-sampling-based perfect simulation
              for {G}ibbs hard-sphere models},
   JOURNAL = {Adv. Appl. Probab.},
    VOLUME = {53},
      YEAR = {2021},
    NUMBER = {3},
     PAGES = {839--885}
}

@book {AG07,
    AUTHOR = {Asmussen, S. and Glynn, P. W.},
     TITLE = {Stochastic Simulation: Algorithms and Analysis},
 PUBLISHER = {Springer},
      YEAR = {2007},
     PAGES = {xiv+476},
   MRCLASS = {65C20 (60G07)},
MRREVIEWER = {John P. Lehoczky},
ADDRESS = {New York}
}

@inproceedings{bb1,
  title={Using {P}oisson processes to model lattice cellular networks},
  author={B{\l}aszczyszyn, B. and Karray, M. K. and Keeler, Holger P.},
  booktitle={2013 Proceedings IEEE INFOCOM},
  pages={773--781},
  year={2013},
  organization={IEEE}
}

@inproceedings{bb2,
  title={A new phase transitions for local delays in {MANET}s},
  author={Baccelli, F. and B{\l}aszczyszyn, B.},
  booktitle={2010 Proceedings IEEE INFOCOM},
  pages={1--9},
  year={2010},
  organization={IEEE}
}

@book{penrose,
        Author = {Penrose, M.~D.},
        Publisher = {Oxford University Press},
        Title = {Random Geometric Graphs},
        Address = {Oxford},
        Year = {2003}
}

@article {gilbert,
        AUTHOR = {Gilbert, E. N.},
        TITLE = {Random plane networks},
        JOURNAL = {J. Soc. Indust. Appl. Math.},
        VOLUME = {9},
        YEAR = {1961},
        PAGES = {533--543}
}

@article{tr1,
  title={Upper tail behavior of the number of triangles in random graphs with constant average degree},
  author={Ganguly, S. and Hiesmayr, E. and Nam, K.},
  journal={Combinatorica},
  pages={699--740},
volume={44},
  year={2024},
  publisher={Springer}
}

@article {tr2,
    AUTHOR = {Stegehuis, C. and Zwart, B.},
     TITLE = {Scale-free graphs with many edges},
   JOURNAL = {Electron. Commun. Probab.},
    VOLUME = {28},
      YEAR = {2023},
     PAGES = {1--11}
}

@article{tr3,
  title={Sparse random graphs with many triangles},
  author={Chakraborty, S. and van der Hofstad, R. and den Hollander, F.},
  journal = {Indagationes Mathematicae},
  year={2026, to appear}
}

@article{comp1,
  title={A large-deviations principle for all the components in a sparse inhomogeneous random graph},
  author={Andreis, L. and K{\"o}nig, W. and Langhammer, H. and Patterson, R. I. A},
  journal={Probab. Theory Related Fields},
  volume={186},
  number={1},
  pages={521--620},
  year={2023}
}

@article{comp2,
  title={Large deviations of the giant in supercritical kernel-based spatial random graphs},
  author={Jorritsma, J. and Komj{\'a}thy, J. and Mitsche, D.},
  journal={Probability Theory and Related Fields},
  year={2026, to appear}
}

@article{comp3,
  title={A large-deviations principle for all the cluster sizes of a sparse {E}rd{\H{o}}s--{R}{\'e}nyi graph},
  author={Andreis, L. and K{\"o}nig, W. and Patterson, R. I. A.},
  journal={Random Structures \& Algorithms},
  volume={59},
  number={4},
  pages={522--553},
  year={2021}
}

@article{lowTails,
        title={Lower large deviations for geometric functionals},
 author={Hirsch, C. and Jahnel, B. and T{\'o}bi{\'a}s, A.},
        JOURNAL = {Electron. Commun. Probab.},
        volume={25},
        pages = "1--12",
        year={2020}
}

@article{harel,
 title={Localization in random geometric graphs with too many edges},
 author={Chatterjee, S. and Harel, M.},
 JOURNAL = {Ann. Probab.},
 VOLUME = {48},
 NUMBER = {1}, 
 PAGES = {574--621},
 YEAR = {2020}
}

@article {cov,
    AUTHOR = {Baccelli, F. and B{\l}aszczyszyn, B.},
     TITLE = {On a coverage process ranging from the {B}oolean model to the
              {P}oisson-{V}oronoi tessellation with applications to wireless
              communications},
   JOURNAL = {Adv. Appl. Probab.},
    VOLUME = {33},
      YEAR = {2001},
    NUMBER = {2},
     PAGES = {293--323}
}

@book{baccelli2009stochastic1,
	Author = {Baccelli, F. and B{\l}aszczyszyn, B.},
	Publisher = {Now Publishers},
	Title = {Stochastic Geometry and Wireless Networks: Volume 1: Theory},
    Address = {Boston, MA},
	Year = {2009}
}

@book{baccelli2009stochastic2,
	Author = {Baccelli, F. and B{\l}aszczyszyn, B.},
	Ovolume = {1},
	Publisher = {Now Publishers},
	Title = {Stochastic Geometry and Wireless Networks: Volume 2: Application},
    Address = {Boston, MA},
	Year = {2009}
}

@book{netCom,
	Author = {Franceschetti, M. and Meester, R.},
    Publisher={Cambridge Series in Statistical and Probabilistic Mathematics},
	Address = {Cambridge University Press},
	Title = {Random Networks for Communication},
    Year={2007}
}

@article{thm,
year = {2009},
volume = {156},
number = {11},
pages = {B1339--B1347},
author = {Thiedmann, R. and Hartnig, C. and Manke, I. and Schmidt, V. and Lehnert, W.},
title = {Local structural characteristics of pore space in {GDL}s of {PEM} fuel cells based on geometric {3D} graphs},
journal = {Journal of The Electrochemical Society}
}

@article{bm,
  title = {Stochastic modeling of molecular charge transport networks},
  author = {Baumeier, B. and Stenzel, O. and Poelking, C. and Andrienko, D. and Schmidt, V.},
  journal = {Phys. Rev. B},
  volume = {86},
  issue = {18},
  pages = {184202},
  numpages = {7},
  year = {2012}
}

@book{diestel2024,
  author    = {Diestel, Reinhard},
  title     = {Graph Theory},
  series    = {Graduate Texts in Mathematics},
  edition   = {6},
  publisher = {Springer},
  address   = {Berlin, Heidelberg},
  year      = {2024},
  isbn      = {978-3-662-70106-5},
  doi       = {10.1007/978-3-662-70107-2}
}

@article{kuratowski1930probleme,
  title={Sur le probleme des courbes gauches en topologie},
  author={Kuratowski, Casimir},
  journal={Fundamenta mathematicae},
  volume={15},
  number={1},
  pages={271--283},
  year={1930}
}
\end{document}